\documentclass[reqno]{amsart}

\usepackage{tikz-cd}
\usetikzlibrary{knots}
\usepackage{spath3}

\usepackage{fullpage,dsfont}

\usepackage{hyperref} % this should come before amsrefs

\usepackage{parskip}
\makeatletter % need this to avoid the conflict between amsthm and parskip
\def\thm@space@setup{%
 \thm@preskip=\parskip \thm@postskip=0pt
}
\def\th@remark{%
  \thm@headfont{\itshape}%
  \normalfont % body font
  \thm@preskip\parskip \thm@postskip=0pt
}
\makeatother

\usepackage[nobysame,alphabetic,initials,msc-links]{amsrefs}

\DefineSimpleKey{bib}{how}
\DefineSimpleKey{bib}{mrclass}
\DefineSimpleKey{bib}{mrnumber}
\DefineSimpleKey{bib}{fjournal}
\DefineSimpleKey{bib}{mrreviewer}

\renewcommand{\PrintDOI}[1]{%
  \href{http://dx.doi.org/#1}{{\tt DOI:#1}}%
}
\renewcommand{\eprint}[1]{#1}
\BibSpec{book}{%
    +{}  {\PrintPrimary}                {transition}
    +{.} { \PrintDate}                  {date}
    +{.} { \textit}                     {title}
    +{.} { }                            {part}
    +{:} { \textit}                     {subtitle}
    +{,} { \PrintEdition}               {edition}
    +{}  { \PrintEditorsB}              {editor}
    +{,} { \PrintTranslatorsC}          {translator}
    +{,} { \PrintContributions}         {contribution}
    +{,} { }                            {series}
    +{,} { \voltext}                    {volume}
    +{,} { }                            {publisher}
    +{,} { }                            {organization}
    +{,} { }                            {address}
    +{,} { }                            {status}
    +{,} { \PrintDOI}                   {doi}
    +{,} { \PrintISBNs}                 {isbn}
    +{}  { \parenthesize}               {language}
    +{}  { \PrintTranslation}           {translation}
    +{;} { \PrintReprint}               {reprint}
    +{.} { }                            {note}
    +{.} {}                             {transition}
    +{}  {\SentenceSpace \PrintReviews} {review}
}
\BibSpec{article}{%
    +{}  {\PrintAuthors}                {author}
    +{,} { \textit}                     {title}
    +{.} { }                            {part}
    +{:} { \textit}                     {subtitle}
    +{,} { \PrintContributions}         {contribution}
    +{.} { \PrintPartials}              {partial}
    +{,} { }                            {journal}
    +{}  { \textbf}                     {volume}
    +{}  { \PrintDatePV}                {date}
    +{,} { \issuetext}                  {number}
    +{,} { \eprintpages}                {pages}
    +{,} { }                            {status}
    +{,} { \PrintDOI}                   {doi}
    +{,} { \eprint}        {eprint}
    +{}  { \parenthesize}               {language}
    +{}  { \PrintTranslation}           {translation}
    +{;} { \PrintReprint}               {reprint}
    +{.} { }                            {note}
    +{.} {}                             {transition}
    +{}  {\SentenceSpace \PrintReviews} {review}
}
\BibSpec{collection.article}{%
    +{}  {\PrintAuthors}                {author}
    +{,} { \textit}                     {title}
    +{.} { }                            {part}
    +{:} { \textit}                     {subtitle}
    +{,} { \PrintContributions}         {contribution}
    +{,} { \PrintConference}            {conference}
    +{}  {\PrintBook}                   {book}
    +{,} { }                            {booktitle}
    +{,} { \PrintDateB}                 {date}
    +{,} { pp.~}                        {pages}
    +{,} { }                            {publisher}
    +{,} { }                            {organization}
    +{,} { }                            {address}
    +{,} { }                            {status}
    +{,} { \PrintDOI}                   {doi}
    +{,} { \eprint}        {eprint}
    +{}  { \parenthesize}               {language}
    +{}  { \PrintTranslation}           {translation}
    +{;} { \PrintReprint}               {reprint}
    +{.} { }                            {note}
    +{.} {}                             {transition}
    +{}  {\SentenceSpace \PrintReviews} {review}
}
\BibSpec{misc}{%
  +{}{\PrintAuthors}  {author}
  +{,}{ \textit}      {title}
  +{.}{ }             {how}
  +{}{ \parenthesize} {date}
  +{,} { available at \eprint}        {eprint}
  +{,}{ available at \url}{url}
  +{,}{ }             {note}
  +{.}{}              {transition}
}

\usepackage{amssymb, amsfonts, amsxtra, amsmath}
\usepackage{mathrsfs}
\usepackage{mathdots}
\usepackage{wasysym}
\usepackage[all]{xy}
\usepackage{bbm}

\numberwithin{equation}{section}

\newtheorem{Theorem}{Theorem}[section]
\newtheorem{Def}[Theorem]{Definition}
\newtheorem{Lem}[Theorem]{Lemma}
\newtheorem{Prop}[Theorem]{Proposition}

\newtheorem{Rem}[Theorem]{Remark}

\newtheorem{Exa}[Theorem]{Example}

\newcommand\bp{\begin{proof}}
\newcommand\ep{\end{proof}}

\mathchardef\mhyph="2D

\DeclareMathOperator{\Ad}{\mathrm{Ad}}

\DeclareMathOperator{\End}{\mathrm{End}}

\DeclareMathOperator{\id}{\mathrm{id}}

\DeclareMathOperator{\twist}{\mathrm{tw}}
\DeclareMathOperator{\Ker}{\mathrm{Ker}}
\DeclareMathOperator{\twistprod}{\underset{\twist}{\times}}

\newcommand{\Z}{\mathbb{Z}}

\newcommand{\opp}{\mathrm{op}}
\newcommand{\triv}{\mathrm{triv}}

\begin{document}

\title{Actions of skew braces and set-theoretic solutions of the reflection equation}

\author{K. De Commer}
\address{Vakgroep wiskunde, Vrije Universiteit Brussel (VUB), B-1050 Brussels, Belgium}
\email{kenny.de.commer@vub.be}
\date{}

\maketitle

\begin{abstract}
A skew brace, as introduced by L. Guarnieri and L. Vendramin, is a set with two group structures interacting in a particular way. When one of the group structures is abelian, one gets back the notion of brace as introduced by W. Rump. Skew braces can be used to construct solutions of the quantum Yang-Baxter equation.  In this article, we introduce a notion of action of a skew brace, and show how it leads to solutions of the closely associated reflection equation. 
\end{abstract}

\section*{Introduction}

The \emph{quantum Yang-Baxter equation} first appeared in the theory of exactly solvable quantum integrable systems \cite{Yan67, Bax72}, but has by now permeated many different areas of mathematics such as Hopf algebras, knot theory and tensor categories. In particular, within the theory of quantum groups the matrix-valued solutions of the quantum Yang-Baxter equation, known as \emph{$R$-matrices}, play a crucial r\^{o}le as a `quasi-commutativity' condition \cite{Dri87}. 

In \cite[Section 9]{Dri92}, V. Drinfeld asked whether there is an interesting theory of \emph{set-theoretic solutions} to the quantum Yang-Baxter equation. This  has recently received a lot of attention \cite{G-IVdB98,ESS99,LYZ00,G-I04,G-IM07,G-IM08,CJO10,G-I12,G-IC12,Deh15}, with a strong impetus coming from the theory of \emph{braces} introduced by W. Rump \cite{Rump07}, see also \cite{CJdR10,Rump14,CJO14},  and more recently \emph{skew braces} \cite{GV17}, see also \cite{SV18,JVA18}.

In a different direction, there has recently also been considerable interest in the \emph{reflection equation}, which first appeared in the study of quantum scattering on the half-line by I.~Cherednik \cite{Che84}. Just as the theory of quantum groups is governed by a universal $R$-matrix satisfying the quantum Yang-Baxter equation, the theory of quantum group actions is closely connected with a universal $K$-matrix satisfying the reflection equation \cite{DKM03,Kol08,KS09,KB15,RV16,Kolb17}. On the categorical level, this corresponds to braidings on module categories over braided tensor categories \cite{tD98,tDH-O98, Bro13,B-ZBJ16}. 

In this article, we make a connection between these two areas by studying solutions to the set-theoretic reflection equation coming from actions of skew braces. As far as we know, the set-theoretic reflection equation only appeared in the physics literature  so far \cite{CZ12,CCZ13}. 

The article is structured as follows. In the \emph{first section}, we recall the definition of skew (left) brace $(A,\circ,\cdot)$ \cite{GV17}, and propose a definition of \emph{skew brace action} as an action of $(A,\circ)$ on a set $X$ together with an equivariant map from $X$ to a universal $(A,\circ)$-space $X_A$. In the \emph{second section}, we consider some examples. In the \emph{third section}, we introduce the \emph{set-theoretic reflection equation}. In the \emph{fourth section}, we recall the notion of a \emph{braiding operator} for a group \cite{LYZ00}. In the \emph{fifth section} we study \emph{braided actions} for groups with a braiding operator, and show how they lead to solutions of the reflection equation. In the \emph{sixth section}, we show how braided actions can be amplified, and subsequently characterized in terms of the amplified actions. In the \emph{seventh section}, we prove the equivalence between skew brace actions and braided actions of groups with a braiding operator. 

\emph{Acknowledgements}: This work was supported by the FWO grant G.0251.15N. The author would like to thank A. Van Antwerpen for bringing \cite{Rump09} to his attention.

\section{Skew braces and actions of skew braces}

\begin{Def}[\cite{GV17}]
A \emph{skew (left) brace}  is a triple $(A,\circ,\cdot)$ such that $\circ$ and $\cdot$ are group structures on the same set $A$, and such that we have the following distributivity law: 
\[
a\circ (b\cdot c) = (a\circ b)\cdot a^{-1}\cdot (a\circ c),
\]
where $a\mapsto a^{-1}$ is the inverse with respect to $\cdot$. 
\end{Def}

Note that the unit $e_A$ of $(A,\circ)$ is then automatically also the unit of $(A,\cdot)$. On the other hand, the inverse $a^{-1}$ of $a$ for $\cdot$ can be distinct from the inverse $\bar{a}$ of $a$ for $\circ$. 

Write 
\begin{equation}\label{EqDefLR}
\lambda_a(b) = a^{-1}\cdot (a\circ b),\qquad \rho_a(b) = (a\circ b)\cdot a^{-1}. 
\end{equation}

\begin{Lem}
Let $(A,\circ,\cdot)$ be a skew brace. Then
\begin{itemize}
\item  $\lambda$ is an action by automorphisms of $(A,\circ)$ on $(A,\cdot)$,
\[
\lambda_a(b\cdot c) = \lambda_a(b)\cdot \lambda_a(c),\qquad \lambda_{a\circ b}(c) = \lambda_a(\lambda_b(c)). 
\]
\item $\rho$ is an action by automorphisms of $(A,\circ)$ on $(A,\cdot)$,
\[
\rho_a(b\cdot c) = \rho_a(b)\cdot  \rho_a(c),\qquad \rho_{a\circ b}(c) = \rho_a(\rho_b(c)). 
\]
\end{itemize}
\end{Lem}
\begin{proof}
The statement for $\lambda$ follows by  \cite[Proposition 1.9]{GV17}. The statement for $\rho$ follows by observing that $(A,\circ,\cdot^{\opp})$, with $a\cdot^{\opp} b =b\cdot a$, is also a skew brace. 
\end{proof}

%We can express condition \eqref{EqComppicirc} as an equivariance condition. 
Let 
\[
X_A = \End(A,\cdot)
\] 
be the monoid of all group endomorphisms of $(A,\cdot)$ with composition as product. Then we can endow $X_A$ (as a set) with the $(A,\circ)$-action
\[
a\diamond \chi := \lambda_a \chi \rho_a^{-1}. 
\]

We propose the following definition for \emph{action} of a skew brace. It will be motivated in the next sections.

\begin{Def}\label{DefAct}
An \emph{action} of a skew brace consists of a triple $(X,\circ,\pi)$ where $X$ is a set endowed with a left action $\circ$ of $(A,\circ)$ and where $\pi$ is an equivariant map 
\begin{equation}\label{EqEquivDef}
\pi: (X,\circ) \rightarrow (X_A,\diamond),\qquad \pi(a\circ x) = a\diamond \pi(x).
\end{equation}
\end{Def}

We write in the following $\pi(x) = \pi_x\in \End(A,\cdot)$. The following proposition shows that the equivariance condition on $\pi$ tells us how to move the products $\cdot$ and $\circ$ through the $\pi_x$. 

\begin{Prop}\label{PropCorrEndo}
Let $(A,\circ,\cdot)$ be a skew brace, and let $(X,\circ)$ be a left action of $(A,\circ)$. Then $(A,\circ,\pi)$ defines an action of $(A,\circ,\cdot)$ if and only if the $\pi_x$ satisfy
\[
\pi_x(a\cdot b) = \pi_x(a)\cdot \pi_x(b)
\]
and
\begin{equation}\label{EqComppicirc}
\pi_x(a\circ b) = \lambda_a(\pi_{\bar{a}\circ x}(b))\cdot \pi_x(a),
\end{equation}
\end{Prop}
\begin{proof}  We can rewrite condition \eqref{EqComppicirc} as 
\begin{equation}\label{EqComppicircOther}
\pi_{a\circ x}(b) = \lambda_a\left(\pi_x(\bar{a}\circ b) \cdot \pi_x(\bar{a})^{-1}\right).
\end{equation}
Assuming the $\pi_x$ are $(A,\cdot)$-endomorphisms, this is equivalent with
\begin{equation}\label{EqLemId5}
\pi_{a\circ x}(b) = \lambda_a(\pi_x((\bar{a}\circ b)\cdot\bar{a}^{-1})) = (\lambda_a\pi_x\rho_{\bar{a}})(b) = (\lambda_a \pi_x \rho_a^{-1})(b),
\end{equation}
which is $\pi_{a\circ x} = a\diamond \pi_x$. 
\end{proof}

Condition \eqref{EqComppicirc} says the $\pi_x$ are a `twisted cocycle family' with respect to $\lambda: (A,\circ) \times (A,\cdot^{\opp}) \rightarrow (A,\cdot^{\opp})$.

%Let us look at some examples. 

%Just as skew braces lead to set-theoretic  solutions of the braid relation or, equivalently, the quantum Yang-Baxter equation, we will show that actions of skew braces lead to set-theoretic solutions of the \emph{reflection equation}. We first study some examples. 

\section{Examples of skew brace actions}

The first example shows that an ordinary action of $(A,\circ)$ can be made into a skew brace action of $(A,\circ,\cdot)$ in a trivial way. 

\begin{Exa}\label{ExaTriv}
Let $(X,\circ)$ be an action of $(A,\circ)$. Then $(X,\circ)$ becomes an action for $(A,\circ,\cdot)$ by putting $\pi_x(a) = e_A$ for all $x\in X$ and $a\in A$. We call this the \emph{trivially extended skew brace action} associated to $(X,\circ)$.
\end{Exa}

For example, in this way the `action of a skew brace on a skew brace' \cite[Definition 2.35]{SV18} can be interpreted as a particular action in our sense. In the case of braces, the same applies to the the brace modules of \cite{Rump09}, which are just left modules\footnote{In \cite{Rump09}, the author works with right modules and right braces, we adapt the definition straightforwardly to the setting of left braces.} over $(A,\circ)$.

\begin{Exa}
By Proposition \ref{PropCorrEndo}, actions of $(A,\circ,\cdot)$ on a one-element set are in one-to-one correspondence with endomorphisms $\chi: (A,\cdot) \rightarrow (A,\cdot)$ such that $\chi$ is a 1-cocycle with respect to $\lambda$ as a group of automorphisms of $(A,\cdot^{\opp})$,
\[
\chi(a\circ b) = \lambda_a(\chi(b))\cdot \chi(a).
\]
These are then precisely the $\diamond$-fixed points in $\End(A,\cdot)$, i.e.~ the maps which intertwine $\rho$ with $\lambda$. For example, $\chi = \id_A$ satisfies this requirement if and only if $(A,\circ,\cdot)$ is a brace, i.e. $\cdot$ is commutative. %In general, there do not need to exist any actions on one-element sets.
\end{Exa}

\begin{Exa}\label{ExaIrrAction}
More generally, let $(B,\circ) \subseteq (A,\circ)$ be a subgroup, and assume that we have an endomorphism $\chi: (A,\cdot) \rightarrow (A,\cdot)$ such that $\chi$ is a cocycle with respect to $\lambda:(B,\circ)\times (A,\cdot^{\opp}) \rightarrow (A,\cdot^{\opp})$, i.e.
\[
\chi(b\circ a) = \lambda_b(\chi(a))\cdot \chi(b),\qquad a\in A, b\in B.
\]
Then it is easily seen that $B\diamond\chi =\chi$. Hence with
\[
\pi(a\circ B) := a\diamond \chi,
\]
we obtain an action $(A/B,\circ,\pi)$ of $(A,\circ,\cdot)$ with respect to the usual $(A,\circ)$-action on $A/B$. We will call such actions \emph{irreducible}. It is clear that, up to isomorphism, any action of $(A,\circ,\cdot)$ is a direct product of actions of the above form. 
\end{Exa}

\begin{Exa}
Returning again to the case of braces, where $\cdot$ is commutative, we see that any action of $(A,\circ)$ on a set $X$ becomes an action of $(A,\circ,\cdot)$ by taking $\pi(x) = \id_A$ for all $x$.  
\end{Exa}

The following example introduces the \emph{universal action}.

\begin{Exa}
Let $(A,\circ,\cdot)$ be a skew brace. We call $(X_A,\diamond,\id)$ the \emph{universal action} of $(A,\circ,\cdot)$. 
\end{Exa}

Next, we present the \emph{standard action} of a skew brace.

\begin{Exa}\label{ExaStandard}
Let $(A,\circ,\cdot)$ be a skew brace. Then $(A,\circ)$ becomes an %non-degenerate 
action for $(A,\circ,\cdot)$ with respect to
\[
\pi_a(b) = \alpha_a(b) := a^{-1}\cdot b\cdot a. 
\]
Indeed, we are to check that $\alpha_{a\circ b} = a\diamond \alpha_b$. But observing that $\rho_a = \alpha_a^{-1}\lambda_a$ and $\alpha_{a\cdot b} = \alpha_b\alpha_a$, we find
\[
a\diamond \alpha_b = \lambda_a\alpha_b \rho_a^{-1} = \lambda_a \alpha_b\lambda_a^{-1} \alpha_a =\alpha_{\lambda_a(b)}\alpha_a = \alpha_{a\cdot \lambda_a(b)} = \alpha_{a\circ b}.
\]
\end{Exa}

A little more generally, we get the following class of inner actions.

\begin{Exa}
Let $(A,\circ,\cdot)$ be a skew brace, and let $(X,\circ)$ be an action of $(A,\circ)$. Assume that 
\[
c: (X,\circ) \rightarrow (A,\circ),\qquad c(a\circ x) = a\circ c(x)
\]
is an equivariant map from $X$ to $A$. Then 
\[
\pi_x(a) = c_x^{-1} \cdot a \cdot c_x
\]
defines an action $(X,\circ,\pi)$ of $(A,\circ,\cdot)$. 
\end{Exa}

The following example treats a particular class of skew braces constructed from groups.

\begin{Exa} Let $(A,\cdot)$ be a group, and consider the associated brace  $(A,\cdot,\cdot)$. Then $\lambda_a  = \id$, while $\rho_a(b) = a\cdot b\cdot a^{-1}$. It follows by Example \ref{ExaIrrAction} that an irreducible action of $(A,\cdot,\cdot)$ is determined by a subgroup $B$ of $(A,\cdot)$  and an endomorphism $\chi \in \End(A,\cdot)$ satisfying
\[
\Ker(\chi) \supseteq \{a\cdot b\cdot a^{-1}\cdot b^{-1}  \mid a\in A,b\in B\}.
\] 
\end{Exa}

\begin{Exa}
By \cite[Example 1.13]{SV18}, the permutation group $(A,\cdot) = S_3$ can be endowed with a unique skew brace structure such that 
\[
\lambda_{\id} = \lambda_{(123)} = \lambda_{132} = \id,\qquad \lambda_{(12)} = \lambda_{(23)} = \lambda_{(31)} = \Ad(23),
\]
with $(A,\circ) \cong \Z/6\Z$ and $(12)$ having order 6 for $\circ$. 

Since $\rho_a(b) = a^{-1}\cdot \lambda_a(b)\cdot a^{-1}$, we then find
\[
\rho_{\id} = \rho_{(23)} = \id,\qquad \rho_{(12)} = \rho_{(123)} = \Ad(123),\qquad \rho_{(31)} = \rho_{(321)} = \Ad(321).
\]
Now as a set, we have 
\[
\End(S_3) \underset{\mu}{\cong} S_3\sqcup \{1,2,3\}\sqcup \{\bullet\},
\]
where, identifying $\{1,2,3\}\cong \Z/3\Z$ and writing $e(t) = 0$ for $t$ even and $e(t) = 1$ for $t$ odd in $S_3$,
\[
\mu_{s}(t) = sts^{-1},\qquad s,t\in S_3,\qquad \mu_m(t) = (m,m+1)^{e(t)},\qquad \mu_{\bullet}(t) = \id,\qquad s,t\in S_3,m\in \{1,2,3\}.
\]
An easy computation then shows that, for $s\in S_3$ and $k\in \{1,2,3\}$, 
\[
(12) \diamond \mu_s = \mu_{(23)s(321)},\qquad (12)\diamond \mu_k = \mu_{(12)k},\qquad (12)\diamond \mu_{\bullet} = \mu_{\bullet}.
\]
We hence find that the universal action $\diamond$ on $\End(S_3)$ decomposes into an orbit $\mu(S_3)$ of order 6, which is equivalent with the standard action of $(A,\circ) = \Z/6\Z$, one two-element orbit $\{\mu_2,\mu_3\}$, and two one-element orbits $\{\mu_1,\mu_{\bullet}\}$. In particular, $(A,\circ,\cdot)$ has a non-trivial action on the one-element set $\{\mu_1\}$.
\end{Exa}

\section{Settheoretic reflection equation}

Diagrams are read from top to bottom.

\begin{Def}
Let $Z$ be a set, and let $r: Z\times Z\rightarrow Z \times Z$ be an invertible map. We say that $r$ satisfies the \emph{braid relation} if
\begin{equation}\label{EqBraidFund}
(r\times \id_Z)(\id_Z\times r)(r\times \id_Z) = (\id_Z\times r)(r\times \id_Z)(\id_Z\times r).
\end{equation}
%We say that $r$ is \emph{non-degenerate} if $r$ is invertible.
\end{Def}

We represent $r$ as a \emph{braid}, 
\begin{figure}[h]
\centering
\begin{tikzpicture}
\node at (0,-1) {$r=$};
\begin{scope}[shift={(1,-0.3)},scale= 0.8]
\begin{knot}[clip width=5, clip radius=8pt]
\strand[thick](0,0)
to [out=down, in=up](2,-2);
\strand[thick](2,0)
to [out=down, in=up](0,-2);
\end{knot}
\end{scope}
\node at (5,-1) {$r^{-1}=$};
\begin{scope}[shift={(6,-0.3)},scale= 0.8]
\begin{knot}[clip width=5, clip radius=8pt]
\strand[thick](0,0)
to [out=down, in=up](2,-2);
\strand[thick](2,0)
to [out=down, in=up](0,-2);
\flipcrossings{1};
\end{knot}
\end{scope}
\end{tikzpicture}
\end{figure}

Then the braid relation corresponds pictorially to the identity of braids

\begin{figure}[h]
\centering
%\leavevmode%\beginpgfgraphicnamed{bkz_conj_fig10}
\begin{tikzpicture}
\begin{scope}%[shift={(0,1)}]
\begin{knot}[clip width=5, clip radius=8pt]
\strand[thick](0,0)
to [out=down, in=up](1,-1)
to [out=down, in=up](2,-2)
to [out =down, in=up](2,-3);
\strand[thick](1,0)
to [out=down, in=up](0,-1)
to [out=down, in=up](0,-2)
to [out =down, in = up](1,-3);
\strand[thick](2,0)
to [out=down, in=up](2,-1)
to [out=down, in=up](1,-2)
to [out =down, in =up](0,-3);
%\flipcrossings{6,2,9,5,11}
\end{knot}
\end{scope}
\node at (3,-1.5) {$=$};
\begin{scope}[shift={(4,0)}]
\begin{knot}[clip width=5, clip radius=8pt]
\strand[thick](0,0)
to [out=down, in=up](0,-1)
to [out=down, in=up](1,-2)
to [out =down, in=up](2,-3);
\strand[thick](1,0)
to [out=down, in=up](2,-1)
to [out=down, in=up](2,-2)
to [out =down, in = up](1,-3);
\strand[thick](2,0)
to [out=down, in=up](1,-1)
to [out=down, in=up](0,-2)
to [out =down, in =up](0,-3);
%\flipcrossings{6,2,9,5,11}
\end{knot}
\end{scope}
\end{tikzpicture}
%\endpgfgraphicnamed
\caption{Braid relation}
\label{fig:BraidRelation}
\end{figure}

If $r$ satisfies the braid relation, then $R(x,y) = r(y,x)$ satisfies the set-theoretic \emph{quantum Yang-Baxter equation}
\[
R_{12}R_{13}R_{23} = R_{23}R_{13}R_{12},
\]
where $R_{ij}$ is $R$ acting on the $i$th and $j$th factor and the identity on the other factors.

A close cousin to the braid relation is the \emph{reflection equation}.

\begin{Def} Let $X,Z$ be sets, and let $r: Z\times Z \rightarrow Z\times Z$ satisfy the braid relation. We say that  a map
\[
k: Z\times X \rightarrow Z\times X
\] 
satisfies the \emph{reflection equation} with respect to $r$, or that $k$ is a solution to the reflection equation for $r$ over the set $X$, if the following maps from $Z\times Z\times X$ to itself are equal:
\begin{equation}\label{EqRE}
(r\times \id_X) (\id_Z\times k) (r\times \id_X) (\id_Z\times k) = (\id_Z\times k) (r\times \id_X) (\id_Z\times k) (r\times \id_X).
\end{equation}
%We say that $k$ is \emph{non-degenerate} if $k$ is invertible.
\end{Def}

We represent $k$ as a full twist around a fixed pole, 

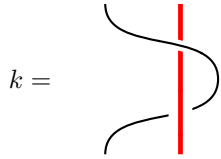
\begin{figure}[h]
\centering
\begin{tikzpicture}
\node at (0,-1) {$k = $};
\begin{scope}[shift={(1,0)}]
\begin{knot}[clip width=5, clip radius = 8pt]
\strand[thick](0,0)
to [out=down, in=up](1.5,-1)
to [out = down,in =up](0,-2);
\strand[line width=1.8,red](1,0)
to [out = down,in = up](1,-2);
\flipcrossings{2}
\end{knot}
\end{scope}
\end{tikzpicture}
\end{figure}

Then the reflection equation is represented pictorially by the \emph{sliding rule}

\begin{figure}[h]
\centering
%\leavevmode%\beginpgfgraphicnamed{bkz_conj_fig10}
\begin{tikzpicture}
\begin{scope}[scale=0.66]
\begin{knot}[clip width=5,clip radius=8pt]
\strand[thick](0,0)
to [out=down, in=up](0,-2)
to [out=down, in=up](1,-3)
to [out =down, in=up](2.5,-4)
to [out =down, in=up](1,-5)
to [out =down, in=up](0,-6);
\strand[thick](1,0)
to [out=down, in=up](2.5,-1)
to [out=down, in=up](1,-2)
to [out =down, in = up](0,-3)
to [out = down, in = up](0,-5)
to [out = down, in = up](1,-6);
\strand[line width=1.8,red](2,0)
to [out=down, in=up](2,-0.4);
\strand[line width = 1.8,red](2,-0.6)
to [out=down, in=up](2,-3)
to [out=down, in=up](2,-4)
to [out=down, in=up](2,-5)
to [out=down, in=up](2,-6);
\flipcrossings{2,4,5}
\end{knot}
\end{scope}
\node at (3,-2) {$=$};
\begin{scope}[shift={(4,0)},scale=0.66]
\begin{knot}[clip width=5,clip radius=8pt]
\strand[thick](0,0)
to [out=down, in=up](1,-1)
to [out=down, in=up](2.5,-2)
to [out =down, in=up](1,-3)
to [out =down, in=up](0,-4)
to [out =down, in=up](0,-6);
\strand[thick](1,0)
to [out=down, in=up](0,-1)
to [out=down, in=up](0,-2)
to [out =down, in = up](0,-3)
to [out = down, in = up](1,-4)
to [out = down,in = up](2.5,-5)
to [out = down, in = right](2.2,-5.25);
\strand[thick](1.8,-5.25)
to [out = left, in = up](1,-6);
\strand[line width=1.8pt,red](2,0)
to [out=down, in=up](2,-6);
\flipcrossings{2,4,6}
\end{knot}
\end{scope}
\end{tikzpicture}
%\endpgfgraphicnamed
\caption{Reflection Equation}
\label{fig:ReflectionEquation}
\end{figure}
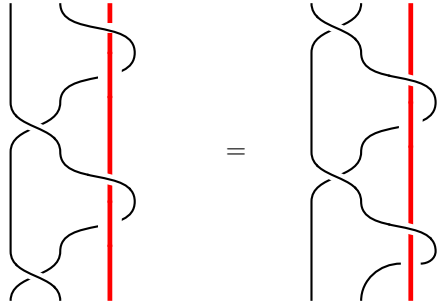

%\newpage

\begin{Exa}
If $r$ satisfies the braid relation, then 
\[
k = r^2: Z\times Z\rightarrow Z\times Z
\]
 satisfies the reflection equation with respect to $r$.
\end{Exa}

\begin{Rem}
%\begin{itemize}
%\item 
%The definition of generalized braided action is inspired by the notion of braided module category over a braided tensor category \cite{Bro13}. 
%\item
In general, the reflection equation appears with an additional `twist' by an automorphism of the group, or with extra parameters. Here we only study the simplest case.
%\end{itemize}
\end{Rem}

\section{Groups with a braiding operator}\label{SecBraidGroup}

\begin{Def}\cite{LYZ00} Let $(A,\circ)$ be a group. A \emph{braiding operator} on $(A,\circ)$  is a map
\[
r:A\times A \rightarrow A\times A,\quad (a,b) \mapsto (a\rhd b,a\lhd b)
\]
such that, with $m_A:A\times A \rightarrow A$ the multiplication and 
\[
\eta:\{\bullet\}\rightarrow A,\quad \bullet \mapsto e_A
\]
the unit map, we have
\begin{equation}\label{EqBraid1}
r(m_A\times \id_A) = (\id_A\times m_A)(r\times \id_A)(\id_A\times r),
\end{equation}
\begin{equation}\label{EqBraidr2}
r(\id_A\times m_A) = (m_A\times \id_A) (\id_A\times r) (r\times \id_A),
\end{equation}
\begin{equation}\label{EqBraidUnit}
r(\id_A\times \eta) = \eta \times \id_A,\qquad r(\eta\times \id_A) = \id_A \times \eta,
\end{equation}
\begin{equation}\label{EqBraidCom}
m_A\circ r = m_A.
\end{equation}
We then call $(A,\circ,r)$ a group with braiding.
\end{Def}
Pictorially, the conditions look as follows:

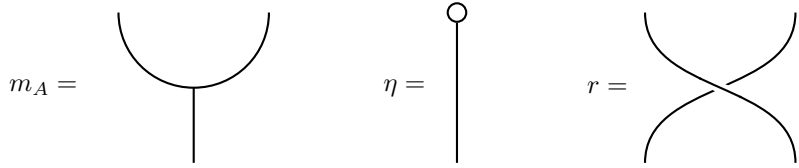
\begin{figure}[h]
\centering
\begin{tikzpicture}
\node at (0,-1) {$m_A = $};
\begin{scope}[shift={(1,0)}]
\begin{knot}[clip width=5, clip radius=8pt]
\strand[thick](0,0)
to [out=down, in=left](1,-1)
to [out=right, in=down](2,0);
\strand[thick](1,-1)
to [out=down, in=up](1,-2);
\end{knot}
\end{scope}
\node at (4.8,-1) {$\eta = $};
\begin{scope}[shift={(5.5,0)}]
\begin{knot}[clip width=5, clip radius=8pt]
\node(A)at(0,0)[circle, draw=black, inner sep=0pt, minimum size=7pt, thick]{};
\strand[thick](A.south)
to [out=down, in=up](0,-2);
\end{knot}
\end{scope}
\node at (7.5,-1) {$r=$};
\begin{scope}[shift={(8,0)}]
\begin{knot}[clip width=5, clip radius=8pt]
\strand[thick](0,0)
to [out=down, in=up](2,-2);
\strand[thick](2,0)
to [out=down, in=up](0,-2);
\end{knot}
\end{scope}
\end{tikzpicture}
\caption{Pictorial representation of multiplication, unit and braid operator}
\label{fig:BraidGroupStructure}
\end{figure}

\begin{figure}[h]
\centering
\begin{tikzpicture}
\begin{scope}%[shift={(1,0)}]
\begin{knot}[clip width=5, clip radius=8pt]
\strand[thick](0,0)
to [out=down, in=left](0.5,-0.5)
to [out=right, in=down](1,0);
\strand[thick](0.5,-0.5)
to [out=down, in=up](1.5,-2);
\strand[thick](1.5,0)
to [out=down, in=up](0.5,-2);
\end{knot}
\end{scope}
\node at (2,-1) {$=$};
\begin{scope}[shift={(3,0)},scale = 0.66]
\begin{knot}[clip width=5, clip radius=8pt]
\strand[thick](0,0)
to [out=down, in=up](0,-1)
to [out=down, in=up](1,-2)
to [out = down, in = left](1.5,-2.5);
\strand[thick](1,-0)
to [out=down, in=up](2,-1)
to [out = down,in = up](2,-2)
to [out = down, in = right](1.5,-2.5);
\strand[thick](2,0)
to [out=down, in=up](1,-1)
to [out=down, in=up](0,-2)
to [out = down, in = up](0,-3);
\strand[thick](1.5,-2.5)
to [out= down,in =up](1.5,-3);
\end{knot}
\end{scope}
\begin{scope}[shift={(8,0)}]
\begin{knot}[clip width=5,clip radius=8pt]
\strand[thick](1.5,0)
to [out=down, in=right](1,-0.5)
to [out=left, in=down](0.5,0);
\strand[thick](1,-0.5)
to [out=down, in=up](0,-2);
\strand[thick](0,0)
to [out=down, in=up](1,-2);
\flipcrossings{1}
\end{knot}
\end{scope}
\node at (10,-1) {$=$};
\begin{scope}[shift={(11,0)},scale = 0.66]
\begin{knot}[clip width=5, clip radius=8pt]
\strand[thick](2,0)
to [out=down, in=up](2,-1)
to [out=down, in=up](1,-2)
to [out = down, in = right](0.5,-2.5);
\strand[thick](1,-0)
to [out=down, in=up](0,-1)
to [out = down,in = up](0,-2)
to [out = down, in = left](0.5,-2.5);
\strand[thick](0,0)
to [out=down, in=up](1,-1)
to [out=down, in=up](2,-2)
to [out = down, in = up](2,-3);
\strand[thick](0.5,-2.5)
to [out= down,in =up](0.5,-3);
\flipcrossings{1,2}
\end{knot}
\end{scope}
\end{tikzpicture}
\caption{Relations \eqref{EqBraid1} and \eqref{EqBraidr2}}
\label{fig:BraidGroupRelations1}
\end{figure}

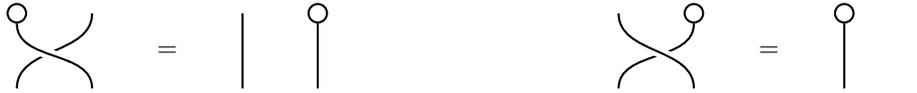
\begin{figure}[h]
\centering
\begin{tikzpicture}
\begin{scope}%[shift={(1,0)}]
\begin{knot}[clip width=5,clip radius=8pt]
\node(A)at(0,0)[circle, draw=black, inner sep=0pt, minimum size=7pt, thick]{};
\strand[thick](A.south)
to [out=down, in=up](1,-1);
\strand[thick](1,0)
to [out=down, in=up](0,-1);
\end{knot}
\end{scope}
\node at (2,-0.5) {$=$};
\begin{scope}[shift={(3,0)}]
\begin{knot}[clip width=5, clip radius=8pt]
\node(A)at(1,0)[circle, draw=black, inner sep=0pt, minimum size=7pt, thick]{};
\strand[thick](0,0)
to [out=down, in=up](0,-1);
\strand[thick](A.south)
to [out=down, in=up](1,-1);
\end{knot}
\end{scope}
\begin{scope}[shift={(8,0)}]
\begin{knot}[clip width=5,clip radius=8pt]
\node(A)at(1,0)[circle, draw=black, inner sep=0pt, minimum size=7pt, thick]{};
\strand[thick](0,0)
to [out=down, in=up](1,-1);
\strand[thick](A.south)
to [out=down, in=up](0,-1);
\end{knot}
\end{scope}
\node at (10,-0.5) {$=$};
\begin{scope}[shift={(11,0)}]
\begin{knot}[clip width=5, clip radius=8pt]
\node(A)at(0,0)[circle, draw=black, inner sep=0pt, minimum size=7pt, thick]{};
\strand[thick](A.south)
to [out=down, in=up](0,-1);
\strand[thick](1,0)
to [out=down, in=up](1,-1);
\end{knot}
\end{scope}
\end{tikzpicture}
\caption{Relations \eqref{EqBraidUnit}}
\label{fig:BraidGroupRelations2}
\end{figure}

\begin{figure}[h]
\centering
\begin{tikzpicture}
\begin{scope}%[shift={(1,0)}]
\begin{knot}[clip width=5, clip radius=8pt]
\strand[thick](0,0)
to [out=down, in=up](1,-1)
to [out=down, in=right](0.5,-1.5);
\strand[thick](1,0)
to [out=down, in=up](0,-1)
to [out = down, in = left](0.5,-1.5);
\strand[thick](0.5,-1.5)
to [out=down, in=up](0.5,-2);
\end{knot}
\end{scope}
\node at (2,-1) {$=$};
\begin{scope}[shift={(3,-0.5)}, scale = 0.66]
\begin{knot}[clip width=5, clip radius=8pt]
\strand[thick](0,0)
to [out=down, in=left](1,-1)
to [out=right, in=down](2,0);
\strand[thick](1,-1)
to [out=down, in=up](1,-2);
\end{knot}
\end{scope}
\end{tikzpicture}
\caption{Relation \eqref{EqBraidCom}}
\label{fig:BraidGroupRelations3}
\end{figure}
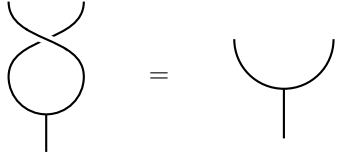

If $(A,\circ,r)$ is a group with braiding, $r: A\times A \rightarrow A \times A$ is automatically invertible and satisfies the braid relation \cite[Corollary 1]{LYZ00}. 

\begin{Exa}\cite[Theorem 9]{LYZ00}
Let 
\[
r: Z\times Z \rightarrow Z\times Z,\quad (w,z)\mapsto (w\rhd z,w\lhd z)
\]
satisfy the braid relation, and let $A(Z)$ be the universal group with generators $\{a_z\mid z\in Z\}$ and relations 
\[
a_wa_z= a_{w\rhd z}a_{w\lhd z}.
\]
Then $A(Z)$ is a group with braiding operator $\widetilde{r}$ uniquely determined by 
\[
\widetilde{r}(a_w,a_z) = (a_{w\rhd z},a_{w\lhd z}).
\]
The study of (semi-)groups of this type was initiated in \cite{G-IVdB98}.
\end{Exa}

The notions of skew brace and group with braiding are equivalent, as the next theorem shows. 

\begin{Theorem}\cite[Theorem 2]{LYZ00},\cite[Remark 3.2]{GV17} \label{TheoEquivBraidGroupBrace}
 Let $(A,\circ)$ be a group. Then there is a one-to-one correspondence between skew brace structures $(A,\circ,\cdot)$ and braiding operators $r$ on $(A,\circ)$, determined by
\begin{equation}\label{EqDefDotBraid}
a\cdot b = a\circ (\bar{a} \rhd b),\qquad a,b\in A,
\end{equation}
or, equivalently,
\begin{equation}\label{EqDefDotBraidEquiv}
a\rhd b = \lambda_a(b).
\end{equation}
\end{Theorem}

To see more conceptually where the dot-product comes from, note that $r$ in particular defines a \emph{matching} \cite{Tak81,Tak03} of $(A,\circ)$ with itself by \eqref{EqBraid1},\eqref{EqBraidr2} and \eqref{EqBraidUnit}. It follows, see e.g. \cite[Proposition 6.2.15]{Maj95}, that we can endow $A\times A$ with the twisted product group structure 
\begin{equation}\label{EqDefTwistProd}
A\twistprod A = (A\times A,\circ_{\twist})
\end{equation}
with
\[
(a_1,a_2) \circ_{\twist} (b_1,b_2) = (a_1\circ (a_2\rhd b_1),(a_2 \lhd b_1)\circ b_2),\qquad e_{A\twistprod A} = (e_A,e_A),\qquad  \overline{(a,b)} = (\bar{b} \rhd \bar{a},\bar{b}\lhd \bar{a}).
\]

\begin{figure}[h]
\centering
\begin{tikzpicture}
\node at (0,-1) {$\circ_{\twist}:= $};
\begin{scope}[shift={(1,0)}]
\begin{knot}[clip width=5, clip radius=8pt]
\strand[thick](0,0)
to [out=down, in=up](0,-1)
to [out=down, in=left](0.5,-1.5);
\strand[thick](1,0)
to [out=down, in=up](2,-1)
to [out=down, in=left](2.5,-1.5);
\strand[thick](2,0)
to [out=down, in = up](1,-1)
to [out=down,in = right](0.5,-1.5);
\strand[thick](3,0)
to [out =down, in = up](3,-1)
to [out = down, in = right](2.5,-1.5);
\strand[thick](0.5,-1.5)
to [out=down,in = up](0.5,-2);
\strand[thick](2.5,-1.5)
to [out=down,in = up](2.5,-2);
\end{knot}
\end{scope}
\end{tikzpicture}
\caption{Twisted product on $A\twistprod A$}
\label{fig:TwistedProductGroupStructure}
\end{figure}

The multiplication map $m_A$ then becomes a group homomorphism,
\[
m_A: (A\times A,\circ_{\twist}) \rightarrow (A,\circ),\quad (a,b)\mapsto a\circ b.
\] 
The kernel $\Ker(m_A) \subseteq A\twistprod A$ consists of all pairs $(a,\bar{a})$, and identifying
\[
\iota: A \overset{\cong}{\rightarrow} \Ker(m_A),\quad a \mapsto (a,\bar{a}),
\]
we see that the induced product on $A$ is precisely the dot-product.

\section{Braided actions}

\begin{Def}
Let $(A,\circ,r)$ be a group with  braiding operator. We call \emph{braided action of $(A,\circ,r)$} a set $X$ together with an action of $(A,\circ)$
\[
m_X: A\times X \rightarrow X,\quad (a,x)\mapsto a\circ x
\] 
and a map
\[
k: A \times X \rightarrow A\times X
\]
such that the following relations are satisfied:
\begin{equation}\label{EqBraid2}
k (m_A\times \id_X) = (m_A\times \id_X) (\id_A \times k) (r\times \id_X)(\id_A \times k), 
\end{equation}
\begin{equation}\label{EqBraidk1}
k (\id_A\times m_X) = (\id_A\times m_X) (r\times \id_X)  (\id_A\times k) (r\times \id_X),
\end{equation}
\begin{equation}\label{EqUnitk}
k (\eta \times \id_X) =  \eta\times \id_X,
\end{equation}
\begin{equation}\label{EqTrivmXk}
m_X k = m_X.
\end{equation}
We call $(A,\circ,r)$ a \emph{generalized braided action} if, in stead of \eqref{EqTrivmXk}, the reflection equation \eqref{EqRE} is satisfied. 
%We call the action \emph{non-degenerate} if $k$ is invertible. 
\end{Def}

These above conditions can again be represented pictorially:

\begin{figure}[h]
\centering
\begin{tikzpicture}
\node at (0,-1) {$m_X = $};
\begin{scope}[shift={(1,0)}]
\begin{knot}[clip width=5, clip radius=8pt]
\strand[thick](0,0)
to [out=down, in=left](1,-1);
\strand[line width=1.8,red](2,0)
to [out = down, in = right](1,-1);
\strand[line width=1.8,red](1,-1)
to [out=down, in=up](1,-2);
\end{knot}
\end{scope}
\node at (4.8,-1) {$k = $};
\begin{scope}[shift={(5.5,0)}]
\begin{knot}[clip width=5, clip radius = 8pt]
\strand[thick](0,0)
to [out=down, in=up](1.5,-1)
to [out = down,in =up](0,-2);
\strand[line width=1.8,red](1,0)
to [out = down,in = up](1,-2);
\flipcrossings{2}
\end{knot}
\end{scope}
\end{tikzpicture}
\caption{Pictorial representation of action and action braid}
\label{fig:BraidActionStructure}
\end{figure}
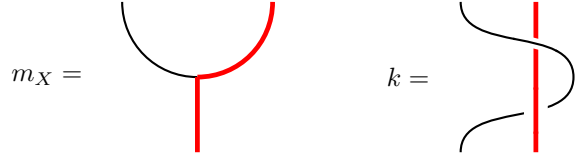

\begin{figure}[h]
\centering
\begin{tikzpicture}
\begin{scope}[shift={(0,-0.5)}]
\begin{knot}[clip width=8, clip radius=8pt]
\strand[thick](0,0)
to [out=down, in=left](0.5,-0.5);
\strand[thick](1,0)
to [out=down, in=right](0.5,-0.5);
\strand[thick](0.5,-0.5)
to [out=down, in=up](2,-1.25)
to [out=down, in=left](1.65,-1.4);
\strand[thick](1.4,-1.4)
to [out=left, in=up](0.5,-2);
\strand[line width=1.8, red](1.5,0)
to [out=down, in=up](1.5,-2);
\end{knot}
\end{scope}
\node at (2.75,-1.5) {$=$};
\begin{scope}[shift={(3.5,0)}, scale = 0.5]
\begin{knot}[clip width=5,clip radius=8pt]
\strand[thick](0,0)
to [out=down, in=up](0,-2)
to [out=down, in=up](1,-3)
to [out =down, in=up](2.5,-4)
to [out =down, in=up](1,-5)
to [out =down, in=right](0.5,-5.5);
\strand[thick](1,0)
to [out=down, in=up](2.5,-1)
to [out = down, in = right](2.2,-1.25);
\strand[thick](1.8,-1.25)
to [out= left, in=up](1,-2)
to [out =down, in = up](0,-3)
to [out = down, in = up](0,-5)
to [out = down, in = left](0.5,-5.5);
\strand[thick](0.5,-5.5)
to [out = down,in = up](0.5,-6);
\strand[line width=1.8,red](2,0)
to [out=down, in=up](2,-0.4);
\strand[line width = 1.8,red](2,-0.6)
to [out=down, in=up](2,-3)
to [out=down, in=up](2,-4)
to [out=down, in=up](2,-5)
to [out=down, in=up](2,-6);
\flipcrossings{3}
\end{knot}
\end{scope}
\begin{scope}[shift={(7,-0.25)}]
\begin{knot}[clip width=5,clip radius=8pt]
\strand[line width = 1.8, red](1.5,0)
to [out=down, in=right](1,-0.5);
\strand[thick](0.5,0)
to [out=down, in=left](1,-0.5);
\strand[thick](0,0)
to [out = down,in =up](0,-0.5)
to [out=down, in=up](1.5,-1.5)
to [out = down, in =up](0,-2.25)
to [out = down, in=up](0,-2.5);
\strand[line width = 1.8,red](1,-0.5)
to [out=down, in=up](1,-2.5);
\flipcrossings{2}
\end{knot}
\end{scope}
\node at (9.25,-1.75) {$=$};
\begin{scope}[shift={(10,0)},scale = 0.66]
\begin{knot}[clip width=5,clip radius=8pt]
\strand[thick](0,0)
to [out=down, in=up](1,-1)
to [out=down, in=up](2.5,-2)
to [out =down, in=up](1,-3)
to [out =down, in=up](0,-4)
to [out =down, in=up](0,-5);
\strand[thick](1,0)
to [out=down, in=up](0,-1)
to [out=down, in=up](0,-2)
to [out =down, in = up](0,-3)
to [out = down, in = up](1,-4)
to [out = down, in = left](1.5,-4.5);
\strand[line width=1.8pt,red](2,0)
to [out=down, in=up](2,-4)
to [out=down, in =right](1.5,-4.5);
\strand[line width=1.8pt,red](1.5,-4.5)
to [out = down, in =  up](1.5,-5);
\flipcrossings{2,4,6}
\end{knot}
\end{scope}

\end{tikzpicture}
\caption{Relations \eqref{EqBraid2} and \eqref{EqBraidk1}}
\label{fig:BraidActionRelations1}
\end{figure}

\begin{figure}[h]
\centering
\begin{tikzpicture}
\begin{scope}%[shift={(5.5,0)}]
\begin{knot}[clip width=5, clip radius = 8pt]
\node(A)at(0,0)[circle, draw=black, inner sep=0pt, minimum size=7pt, thick]{};
\strand[thick](A.south)
to [out=down, in=up](1.5,-1)
to [out = down,in =up](0,-2);
\strand[line width=1.8,red](1,0)
to [out = down,in = up](1,-2);
\flipcrossings{2}
\end{knot}
\end{scope}
\node at (2.75,-1) {$ = $};
\begin{scope}[shift={(3.5,0)}]
\begin{knot}[clip width=5, clip radius = 8pt]
\node(A)at(0,0)[circle, draw=black, inner sep=0pt, minimum size=7pt, thick]{};
\strand[thick](A.south)
to [out=down, in=up](0,-2);
\strand[line width=1.8,red](1,0)
to [out = down,in = up](1,-2);
\flipcrossings{2}
\end{knot}
\end{scope}
\begin{scope}[shift={(7,0)},scale = 0.66]
\begin{knot}[clip width=5, clip radius = 8pt]
\strand[thick](0,0)
to [out=down, in=up](1.5,-1)
to [out = down,in =up](0,-2)
to [out = down,in =left](0.5,-2.5);
\strand[line width=1.8,red](1,0)
to [out = down,in = up](1,-0.4);
\strand[line width=1.8,red](1,-0.65)
to [out = down,in = up](1,-2)
to [out = down, in = right](0.5,-2.5);
\strand[line width=1.8,red](0.5,-2.5)
to [out = down, in = up](0.5,-3);
\flipcrossings{1}
\end{knot}
\end{scope}
\node at (9,-1) {$ = $};
\begin{scope}[shift={(9.5,0)}]
\begin{knot}[clip width=5, clip radius=8pt]
\strand[thick](0,0)
to [out=down, in=left](1,-1);
\strand[line width=1.8,red](2,0)
to [out = down, in = right](1,-1);
\strand[line width=1.8,red](1,-1)
to [out=down, in=up](1,-2);
\end{knot}
\end{scope}
\end{tikzpicture}
\caption{Relations \eqref{EqUnitk} and \eqref{EqTrivmXk}}
\label{fig:BraidActionRelations2}
\end{figure}

\newpage

\begin{Rem}
%\begin{itemize}
%\item 
The definition of (generalized) braided action is inspired by the notion of braided module category over a braided tensor category \cite{Bro13}, see also the earlier work \cite{tD98,tDH-O98} on cylinder twists.
%\item
%In general, the reflection equation appears with an additional `twist' by an automorphism of the group, or with extra parameters. Here we only study the simplest case.
%\end{itemize}
\end{Rem}

\begin{Exa}
Let $(A,\circ,r)$ be a group with braiding, and put $k=r^2$. Then $(A,m_A,k)$ is a braided action for $(A,\circ,r)$, using for \eqref{EqBraid2} that $r$ satisfies the braid relation. 
\end{Exa}

\begin{Lem}\label{LemkFromBraidAction}
Let $(X,m_X,k)$ be a braided action for a group with braiding $(A,\circ,r)$. Then $k$ satisfies the reflection equation with respect to $r$.
\end{Lem}
In other words, any braided action is a generalized braided action.
\begin{proof}
We use the technique of \cite[Theorem 1]{LYZ00}. Let $(a,b,x)\in A\times A \times X$, and let $(a_1,b_1,x_1)$ and $(a_2,b_2,x_2)$ be the result of applying respectively the left hand side and the right hand side of \eqref{EqRE} to $(a,b,x)$. It is then sufficient to show that 
\[
(a_1\circ b_1,x_1) = (a_2\circ b_2,x_2),\qquad (a_1,b_1\circ x_1)= (a_2,b_2\circ x_2).
\]
These identities can be proven pictorially as follows:  

\newpage

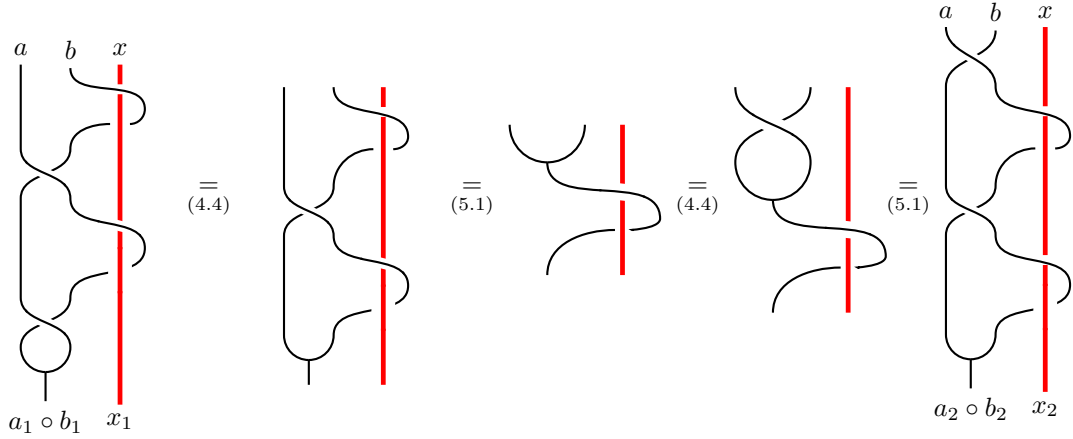
\begin{figure}[h]
\centering
%\leavevmode%\beginpgfgraphicnamed{bkz_conj_fig10}
\begin{tikzpicture}
\begin{scope}[scale=0.66]
\begin{knot}[clip width=5,clip radius=8pt]
\node(A)at(0,0){$a$};
\node(B)at(1,0){$b$};
\node(X)at(2,0){$x$};
\node(A')at(0.5,-7.5){$a_1\circ b_1$};
\node(B')at(2,-7.5){$x_1$};
\strand[thick](A.south)
to [out=down, in=up](0,-2)
to [out=down, in=up](1,-3)
to [out =down, in=up](2.5,-4)
to [out =down, in=up](1,-5)
to [out =down, in=up](0,-6)
to [out = down, in = left](0.5,-6.5);
\strand[thick](B.south)
to [out=down, in=up](2.5,-1.2)
to [out = down, in = right](2.2,-1.5);
\strand[thick](1.8,-1.5)
to [out=left, in=up](1,-2)
to [out =down, in = up](0,-3)
to [out = down, in = up](0,-5)
to [out = down, in = up](1,-6)
to [out = down, in = right](0.5,-6.5);
\strand[thick](0.5,-6.5)
to [out = down, in = up](A'.north);
\strand[line width=1.8,red](X.south)
to [out=down, in=up](2,-0.7);
\strand[line width = 1.8,red](2,-0.9)
to [out=down, in=up](2,-3)
to [out=down, in=up](2,-4)
to [out=down, in=up](2,-5)
to [out=down, in=up](B'.north);
\flipcrossings{2,4,5}
\end{knot}
\end{scope}
\node at (2.5,-2) {$\underset{\eqref{EqBraidCom}}{=}$};
\begin{scope}[shift={(3.5,-0.5)}, scale = 0.66]
\begin{knot}[clip width=5,clip radius=8pt]
\strand[thick](0,0)
to [out=down, in=up](0,-2)
to [out=down, in=up](1,-3)
to [out =down, in=up](2.5,-4)
to [out =down, in=up](1,-5)
to [out =down, in=right](0.5,-5.5);
\strand[thick](1,0)
to [out=down, in=up](2.5,-1)
to [out = down, in = right](2.2,-1.25);
\strand[thick](1.8,-1.25)
to [out= left, in=up](1,-2)
to [out =down, in = up](0,-3)
to [out = down, in = up](0,-5)
to [out = down, in = left](0.5,-5.5);
\strand[thick](0.5,-5.5)
to [out = down,in = up](0.5,-6);
\strand[line width=1.8,red](2,0)
to [out=down, in=up](2,-0.4);
\strand[line width = 1.8,red](2,-0.6)
to [out=down, in=up](2,-3)
to [out=down, in=up](2,-4)
to [out=down, in=up](2,-5)
to [out=down, in=up](2,-6);
\flipcrossings{3}
\end{knot}
\end{scope}
\node at (6,-2) {$\underset{\eqref{EqBraid2}}{=}$};
\begin{scope}[shift={(6.5,-1)}]
\begin{knot}[clip width=8, clip radius=8pt]
\strand[thick](0,0)
to [out=down, in=left](0.5,-0.5);
\strand[thick](1,0)
to [out=down, in=right](0.5,-0.5);
\strand[thick](0.5,-0.5)
to [out=down, in=up](2,-1.25)
to [out=down, in=left](1.65,-1.4);
\strand[thick](1.4,-1.4)
to [out=left, in=up](0.5,-2);
\strand[line width=1.8, red](1.5,0)
to [out=down, in=up](1.5,-2);
\end{knot}
\end{scope}
\node at (9,-2) {$\underset{\eqref{EqBraidCom}}{=}$};
\begin{scope}[shift={(9.5,-0.5)}]
\begin{knot}[clip width=8, clip radius=8pt]
\strand[thick](1,0)
to [out = down, in = up](0,-1)
to [out=down, in=left](0.5,-1.5);
\strand[thick](0,0)
to [out = down,in = up](1,-1)
to [out=down, in=right](0.5,-1.5);
\strand[thick](0.5,-1.5)
to [out=down, in=up](2,-2.25)
to [out=down, in=left](1.65,-2.4);
\strand[thick](1.4,-2.4)
to [out=left, in=up](0.5,-3);
\strand[line width=1.8, red](1.5,0)
to [out=down, in=up](1.5,-3);
\flipcrossings{1}
\end{knot}
\end{scope}
\node at (11.8,-2) {$\underset{\eqref{EqBraid2}}{=}$};
\begin{scope}[shift={(12.3,0.5)}, scale = 0.66]
\begin{knot}[clip width=5,clip radius=8pt]
\node(A)at(0,0){$a$};
\node(B)at(1,0){$b$};
\node(X)at(2,0){$x$};
\node(A')at(0.5,-8){$a_2\circ b_2$};
\node(B')at(2,-8){$x_2$};
\strand[thick](B.south)
to [out = down, in = up](0,-1.5)
to [out=down, in=up](0,-3.5)
to [out=down, in=up](1,-4.5)
to [out =down, in=up](2.5,-5.5)
to [out =down, in=up](1,-6.5)
to [out =down, in=right](0.5,-7);
\strand[thick](A.south)
to [out = down, in =up](1,-1.5)
to [out=down, in=up](2.5,-2.5)
to [out = down, in = right](2.2,-2.75);
\strand[thick](1.8,-2.75)
to [out= left, in=up](1,-3.5)
to [out =down, in = up](0,-4.5)
to [out = down, in = up](0,-6.5)
to [out = down, in = left](0.5,-7);
\strand[thick](0.5,-7)
to [out = down,in = up](A'.north);
\strand[line width=1.8,red](X.south)
to [out=down, in=up](2,-1.9);
\strand[line width = 1.8,red](2,-2.1)
to [out=down, in=up](2,-2.5)
to [out=down, in=up](2,-4.9);
\strand[line width = 1.8,red](2,-5.1)
to [out=down, in=up](B'.north);
\flipcrossings{1,3}
\end{knot}
\end{scope}
\end{tikzpicture}
\end{figure}

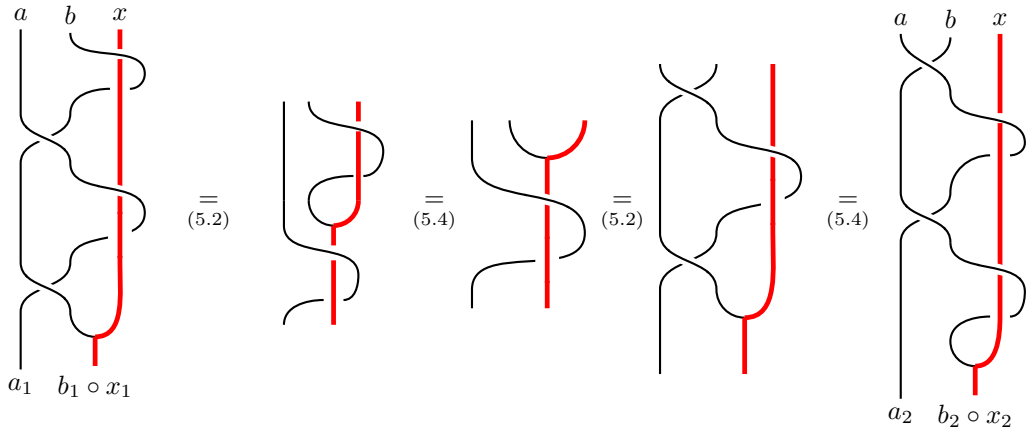
\begin{figure}[h]
\centering
%\leavevmode%\beginpgfgraphicnamed{bkz_conj_fig10}
\begin{tikzpicture}
\begin{scope}[scale=0.66,shift={(0,1)}]
\begin{knot}[clip width=5,clip radius=8pt]
\node(A)at(0,0){$a$};
\node(B)at(1,0){$b$};
\node(X)at(2,0){$x$};
\node(A')at(0,-7.5){$a_1$};
\node(B')at(1.5,-7.5){$b_1\circ x_1$};
\strand[thick](A.south)
to [out=down, in=up](0,-2)
to [out=down, in=up](1,-3)
to [out =down, in=up](2.5,-4)
to [out =down, in=up](1,-5)
to [out =down, in=up](0,-6)
to [out = down, in = up](A'.north);
\strand[thick](B.south)
to [out=down, in=up](2.5,-1.2)
to [out = down, in = right](2.2,-1.5);
\strand[thick](1.8,-1.5)
to [out=left, in=up](1,-2)
to [out =down, in = up](0,-3)
to [out = down, in = up](0,-5)
to [out = down, in = up](1,-6)
to [out = down, in = left](1.5,-6.5);
\strand[line width=1.8,red](X.south)
to [out=down, in=up](2,-0.7);
\strand[line width = 1.8,red](2,-0.9)
to [out=down, in=up](2,-3)
to [out=down, in=up](2,-4)
to [out=down, in=up](2,-5)
to [out = down,in = right](1.5,-6.5);
\strand[line width = 1.8,red](1.5,-6.5)
to [out=down, in=up](B'.north);
\flipcrossings{2,4,5}
\end{knot}
\end{scope}
\node at (2.5,-2) {$\underset{\eqref{EqBraidk1}}{=}$};
\begin{scope}[shift={(3.5,-0.5)}, scale = 0.66]
\begin{knot}[clip width=5,clip radius=8pt]
\strand[thick](0.5,0)
to [out=down, in=up](2,-1)
to [out = down, in = right](1.6,-1.5);
\strand[thick](1.4,-1.5)
to [out = left,in =up](0.5,-2);
\strand[line width=1.8,red](1.5,0)
to [out = down, in = up](1.5,-0.4);
\strand[line width=1.8,red](1.5,-0.6)
to [out = down,in = up](1.5,-2);
\strand[thick](0,0)
to [out = down, in = up](0,-2);
\strand[line width = 1.8, red](1.5,-2)
to [out=down, in=right](1,-2.5);
\strand[thick](0.5,-2)
to [out=down, in=left](1,-2.5);
\strand[thick](0,-2)
to [out = down,in =up](0,-2.5)
to [out=down, in=up](1.5,-3.5)
to [out = down,in = right](1.2,-4);
\strand[thick](0.8,-4)
to [out = left, in=up](0,-4.5);
\strand[line width = 1.8,red](1,-2.5)
to [out = down, in = up](1,-2.9);
\strand[line width = 1.8,red](1,-3.2)
to [out=down, in=up](1,-4.5);
\flipcrossings{2}
\end{knot}
\end{scope}
\node at (5.5,-2) {$\underset{\eqref{EqTrivmXk}}{=}$};
\begin{scope}[shift={(6,-0.75)}]
\begin{knot}[clip width=5,clip radius=8pt]
\strand[line width = 1.8, red](1.5,0)
to [out=down, in=right](1,-0.5);
\strand[thick](0.5,0)
to [out=down, in=left](1,-0.5);
\strand[thick](0,0)
to [out = down,in =up](0,-0.5)
to [out=down, in=up](1.5,-1.5)
to [out = down, in =up](0,-2.25)
to [out = down, in=up](0,-2.5);
\strand[line width = 1.8,red](1,-0.5)
to [out=down, in=up](1,-2.5);
\flipcrossings{2}
\end{knot}
\end{scope}
\node at (8,-2) {$\underset{\eqref{EqBraidk1}}{=}$};
\begin{scope}[shift={(8.5,0)},scale=0.75]
\begin{knot}[clip width=5,clip radius=8pt]
\strand[thick](0,0)
to [out=down, in=up](1,-1)
to [out =down, in=up](2.5,-2)
to [out =down, in=up](1,-3)
to [out =down, in=up](0,-4)
to [out = down, in = up](0,-5.5);
\strand[thick](1,0)
to [out =down, in = up](0,-1)
to [out = down, in = up](0,-3)
to [out = down, in = up](1,-4)
to [out = down, in = left](1.5,-4.5);
\strand[line width=1.8,red](2,0)
to [out=down, in=up](2,-1)
to [out=down, in=up](2,-2)
to [out=down, in=up](2,-3)
to [out = down,in = right](1.5,-4.5);
\strand[line width = 1.8,red](1.5,-4.5)
to [out=down, in=up](1.5,-5.5);
\flipcrossings{2,4,5}
\end{knot}
\end{scope}
\node at (11,-2) {$\underset{\eqref{EqTrivmXk}}{=}$};
\begin{scope}[shift={(11.7,0.6)}, scale = 0.66]
\begin{knot}[clip width=5,clip radius=8pt]
\node(A)at(0,0){$a$};
\node(B)at(1,0){$b$};
\node(X)at(2,0){$x$};
\node(A')at(0,-8){$a_2$};
\node(B')at(1.5,-8){$b_2\circ x_2$};
\strand[thick](B.south)
to [out = down, in = up](0,-1.5)
to [out=down, in=up](0,-3.5)
to [out=down, in=up](1,-4.5);
\strand[thick](A.south)
to [out = down, in =up](1,-1.5)
to [out=down, in=up](2.5,-2.5)
to [out = down, in = right](2.2,-2.75);
\strand[thick](1.8,-2.75)
to [out= left, in=up](1,-3.5)
to [out = down,in =up](0,-4.5)
to [out =down, in = up](A'.north);
\strand[line width=1.8,red](X.south)
to [out=down, in=up](2,-1.9);
\strand[line width = 1.8,red](2,-2.1)
to [out=down, in=up](2,-2.5)
to [out=down, in=up](2,-6)
to [out = down, in = right](1.5,-7);
\strand[thick](1,-4.5)
to [out=down, in=up](2.5,-5.5)
to [out = down, in = right](2.2,-6);
\strand[thick](1.8,-6)
to [out = left,in =up](1,-6.5)
to [out = down, in = left](1.5,-7);
\strand[line width = 1.8,red](1.5,-7)
to [out=down, in=up](B'.north);
\flipcrossings{1,3}
\end{knot}
\end{scope}
\end{tikzpicture}
\end{figure}

\end{proof}

We will show in Section \ref{SecProofEquiv} that an action of a skew brace $(A,\circ,\cdot)$ is equivalent to a braided action of $(A,\circ,r)$ for $r$ the braiding operator associated to the skew brace structure. This shows, by means of Lemma \ref{LemkFromBraidAction}, that actions of skew braces lead to solutions of the reflection equation.

\section{Product actions}

In the following, we fix a group with braiding $(A,\circ,r)$. Given a generalized braided action $(X,m_X,k_X)$, we will construct three generalized braided actions on $A\times X$. 

%skew brace $(A,\circ,\cdot)$ with associated braiding 
%\[
%r(a,b) = (a\rhd b,a\lhd b) = (\lambda_a(b), \overline{\lambda_a(b)}\circ a \circ b).
%\] 

\begin{Prop}\label{PropAmpAction}
Let $(X,m_X,k_X)$ be a generalized braided action of $(A,\circ,r)$. Then $Y = A\times X$ has the following generalized braided actions of $(A,\circ,r)$, all with the same action braid 
\begin{equation}\label{EqDefkamp}
k_Y = (r\times \id_X)(\id_A\times k_X)(r\times \id_X): A\times (A\times X) \rightarrow A\times (A\times X),
\end{equation}
namely:
\begin{itemize}
\item the \emph{absorbing $A$-extension}
\begin{equation}\label{EqDefTrivAct}
m_Y^{\triv} = m_A\times \id_X,
\end{equation}
\item the \emph{$k$-twisted absorbing $A$-extension}
\begin{equation}\label{EqDefMu}
m_Y^k = (m_A\times \id_X)(\id_A \times k_X)(r\times \id_X): A\times (A\times X) \rightarrow A\times X 
\end{equation}
\item the \emph{$r$-twisted $A$-amplification} of $m_X$, 
\begin{equation}\label{EqDefrTwistAct}
m_Y^r = (\id_X \times m_X)(r\times \id_X).
\end{equation}
\end{itemize}
\end{Prop}

We will present the product actions and action braid as follows:

\begin{figure}[h]
\centering
\begin{tikzpicture}
\node at (0,-1) {$m_Y^{\triv} = $};
\begin{scope}[shift={(1,-0.5)},scale = 0.6]
\begin{knot}[clip width=5, clip radius=8pt]
\strand[thick,dotted](0,0)
to [out=down, in=left](1,-1);
\strand[line width=1.8,green](2,0)
to [out = down, in = right](1,-1);
\strand[line width=1.8,green](1,-1)
to [out=down, in=up](1,-2);
\end{knot}
\end{scope}
\node at (2.5,-1) {$:=$};
\begin{scope}[shift={(3,-0.6)}]
\begin{knot}[clip width=5,clip radius=8pt]
\strand[thick](0,0)
to [out=down, in=left](0.5,-0.5);
\strand[thick](1,0)
to [out = down, in =right](0.5,-0.5);
\strand[thick](0.5,-0.5)
to [out =down, in = up](0.5,-1);
\strand[line width=1.8,red](1.5,0)
to [out=down, in=up](1.5,-1);
\flipcrossings{3}
\end{knot}
\end{scope}

\node at (6,-1) {$m_Y^k = $};
\begin{scope}[shift={(6.7,-0.5)},scale = 0.6]
\begin{knot}[clip width=5, clip radius=8pt]
\strand[thick](0,0)
to [out=down, in=left](1,-1);
\strand[line width=1.8,green](2,0)
to [out = down, in = right](1,-1);
\strand[line width=1.8,green](1,-1)
to [out=down, in=up](1,-2);
\end{knot}
\end{scope}
\node at (8.4,-1) {$:=$};
\begin{scope}[shift={(9,0)}, scale = 0.5]
\begin{knot}[clip width=5,clip radius=8pt]
\strand[thick](0,0)
to [out=down, in=up](1,-1)
to [out =down, in=up](2.5,-2)
to [out =down, in=up](1,-3)
to [out =down, in=right](0.5,-3.5);
\strand[thick](1,0)
to [out =down, in = up](0,-1)
to [out = down, in = up](0,-3)
to [out = down, in = left](0.5,-3.5);
\strand[thick](0.5,-3.5)
to [out = down,in = up](0.5,-4);
\strand[line width=1.8,red](2,0)
to [out=down, in=up](2,-4);
\flipcrossings{3}
\end{knot}
\end{scope}
\node at (11.5,-1) {$m_X^r = $};
\begin{scope}[shift={(12.2,-0.5)},scale =0.6]
\begin{knot}[clip width=5, clip radius=8pt]
\strand[thick,dashed](0,0)
to [out=down, in=left](1,-1);
\strand[line width=1.8,green](2,0)
to [out = down, in = right](1,-1);
\strand[line width=1.8,green](1,-1)
to [out=down, in=up](1,-2);
\end{knot}
\end{scope}
\node at (13.7,-1) {$:=$};
\begin{scope}[shift={(14.2,-0.5)},scale =0.6]
\begin{knot}[clip width=5, clip radius=8pt]
\strand[thick](0,0)
to [out=down, in=up](1,-1)
to [out = down, in = left](1.5,-1.5);
\strand[thick](1,0)
to[out = down, in = up](0,-1)
to [out = down, in = up](0,-2);
\strand[line width=1.8,red](2,0)
to [out = down, in = up](2,-1)
to [out = down, in = right](1.5,-1.5) ;
\strand[line width = 1.8,red](1.5,-1.5)
to [out =down, in = up](1.5,-2);
\end{knot}
\end{scope}
\end{tikzpicture}
\caption{Pictorial representation of product actions}
\label{fig:ProdActionStructure}
\end{figure}
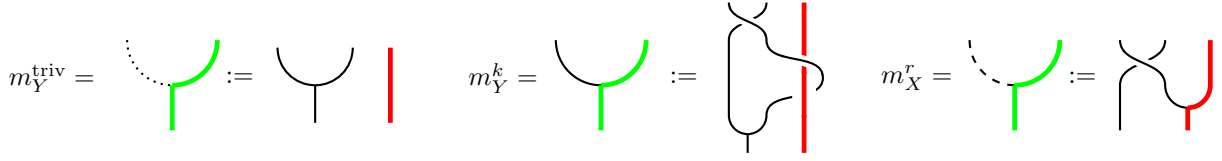

\begin{figure}[h]
\centering
\begin{tikzpicture}
\node at (7,-1) {$k_Y = $};
\begin{scope}[shift={(7.5,0)}]
\begin{knot}[clip width=5, clip radius = 8pt]
\strand[thick](0,0)
to [out=down, in=up](1.5,-1)
to [out = down,in =up](0,-2);
\strand[line width=1.8,green](1,0)
to [out = down,in = up](1,-2);
\flipcrossings{2}
\end{knot}
\end{scope}
\node at (9.8,-1) {$:= $};
\begin{scope}[shift={(10.75,0.25)},scale = 0.66]
\begin{knot}[clip width=5,clip radius=8pt]
\strand[thick](0,0)
to [out=down, in=up](1,-1)
to [out=down, in=up](2.5,-2)
to [out = down,in=right](2.2,-2.5);
\strand[thick](1.8,-2.5)
to [out =left, in=up](1,-3)
to [out =down, in=up](0,-4);
\strand[thick](1,0)
to [out=down, in=up](0,-1)
to [out=down, in=up](0,-2)
to [out =down, in = up](0,-3)
to [out = down, in = up](1,-4);
\strand[line width=1.8pt,red](2,0)
to [out = down, in = up](2,-1.4);
\strand[line width = 1.8pt,red](2,-1.6)
to [out=down, in=up](2,-4);
\flipcrossings{2,4,6}
\end{knot}
\end{scope}
\end{tikzpicture}
\caption{Pictorial representation of  product action braid}
\label{fig:ProdBraidStructure}
\end{figure}
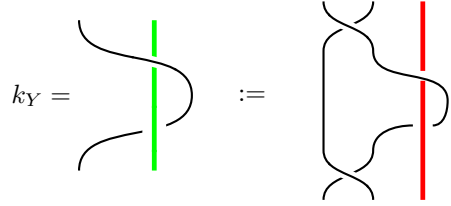

\begin{proof}

Trivially, $m_Y^{\triv}$ is an action. The action property for $m_Y^k$ follows from the following computation:

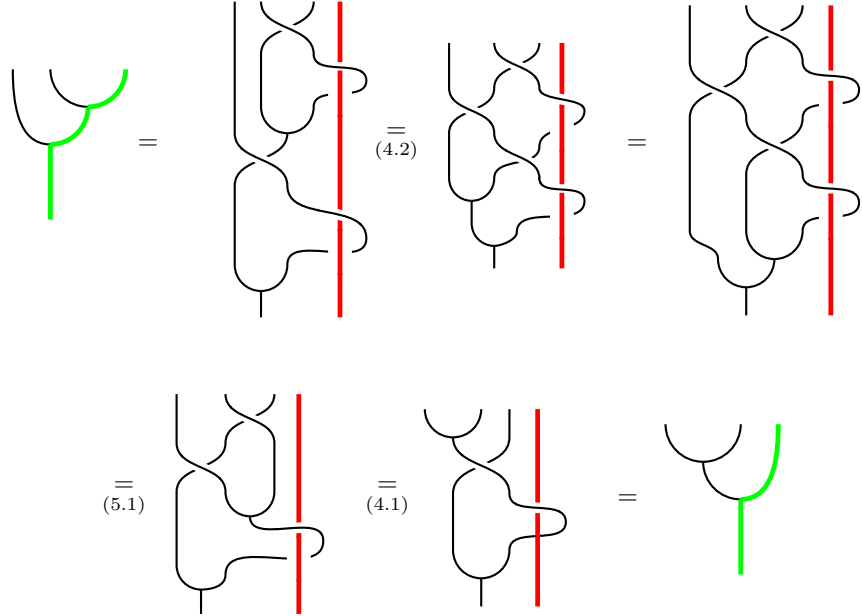
\begin{figure}[h]
\centering
\begin{tikzpicture}
\begin{scope}
\begin{knot}[clip width=5, clip radius=8pt]
\strand[thick](0,0)
to [out=down, in=left](0.5,-1);
\strand[thick](0.5,0)
to [out = down,in = left](1,-0.5);
\strand[line width = 1.8,green](1.5,0)
to [out = down, in = right](1,-0.5);
\strand[line width=1.8,green](1,-0.5)
to [out = down, in = right](0.5,-1);
\strand[line width=1.8,green](0.5,-1)
to [out=down, in=up](0.5,-2);
\end{knot}
\end{scope}
\node at (1.8,-1) {$= $};
\begin{scope}[shift={(2.6,0.9)},scale=0.7]
\begin{knot}[clip width=5, clip radius=8pt]
\strand[thick](0.5,0)
to [out=down, in=up](0.5,-2.5)
to [out = down, in =up](1.5,-3.5)
to [out = down, in = up](3,-4.5)
to [out = down, in = up](1.5,-5)
to [out = down, in = right](1,-5.5);
\strand[thick](1,-5.5)
to [out = down, in = up](1,-6);
\strand[thick](1,0)
to [out = down, in  =up](2,-1)
to [out = down, in = up](3,-1.5)
to [out = down, in =up](2,-2)
to [out = down, in = right](1.5,-2.5);
\strand[thick](2,0)
to [out = down, in = up](1,-1)
to [out = down, in = up](1,-2)
to [out = down, in = left](1.5,-2.5);
\strand[thick](1.5,-2.5)
to [out = down,in = up](0.5,-3.5)
to [out = down, in = up](0.5,-5)
to [out = down, in = left](1,-5.5);
\strand[line width=1.8,red](2.5,0)
to [out=down, in=up](2.5,-6);
\flipcrossings{3,6}
\end{knot}
\end{scope}
\node at (5.1,-1) {$\underset{\eqref{EqBraidr2}}{=}$};
\begin{scope}[shift={(5.8,0.35)},scale=0.6]
\begin{knot}[clip width=5, clip radius=8pt]
\strand[thick](0,0)
to [out=down, in=up](0,-1)
to [out = down, in =up](1,-2)
to [out = down, in = up](2,-3)
to [out = down, in = up](3,-3.5)
to [out = down, in = up](1.5,-4.2)
to [out = down, in = right, looseness=0.7](1,-4.5);
\strand[thick](1,-4.5)
to [out = down, in = up](1,-5);
\strand[thick](1,0)
to [out = down, in  =up](2,-1)
to [out = down, in = up,looseness=0.7](3,-1.5)
to [out = down, in =up](2,-2.3)
to [out = down, in = up](1,-3)
to [out = down, in = right](0.5,-3.5);
\strand[thick](0.5,-3.5)
to [out = down, in = up](0.5,-4)
to [out = down, in = left](1,-4.5);
\strand[thick](2,0)
to [out = down, in = up](1,-1)
to [out = down, in = up](0,-2)
to [out = down, in = up](0,-3)
to [out = down, in = left](0.5,-3.5);
\strand[line width=1.8,red](2.5,0)
to [out=down, in=up](2.5,-5);
\flipcrossings{7,4}
\end{knot}
\end{scope}
\node at (8.3,-1) {$= $};
\begin{scope}[shift={(9,0.85)},scale=0.75]
\begin{knot}[clip width=5, clip radius=8pt]
\strand[thick](0,0)
to [out=down, in=up](0,-1)
to [out = down, in =up](1,-2)
to [out = down, in = up](2,-3)
to [out = down, in = up](3,-3.5)
to [out = down, in = up](2,-4)
to [out = down, in = right](1.5,-4.5);
%to [out = down, in = right](1,-5.5);
\strand[thick](1.5,-4.5)
to [out = down, in = right](1,-5);
\strand[thick](1,0)
to [out = down, in  =up](2,-1)
to [out = down, in = up](3,-1.5)
to [out = down, in =up](2,-2)
to [out = down, in = up](1,-3)
to [out = down, in = up](1,-4)
to [out = down, in = left](1.5,-4.5);
\strand[thick](1,-5)
to [out = down, in = up](1,-5.5);
\strand[thick](2,0)
to [out = down, in = up](1,-1)
to [out = down, in = up](0,-2)
to [out = down, in =up](0,-4)
to [out = down, in = up](0.5,-4.5)
to [out = down, in = left](1,-5);
\strand[line width=1.8,red](2.5,0)
to [out=down, in=up](2.5,-5.5);
\flipcrossings{7,4}
\end{knot}
\end{scope}
\end{tikzpicture}
\end{figure}

\begin{figure}[h]
\centering
\begin{tikzpicture}
\node at (0,-1.4) {$\underset{\eqref{EqBraid2}}{=} $};
\begin{scope}[shift={(0.7,0)},scale=0.65]
\begin{knot}[clip width=5, clip radius=8pt]
\strand[thick](0,0)
to [out=down, in= up](0,-1)
to [out = down, in =up](1,-2)
to [out = down,in = left](1.5,-2.5);
\strand[thick](1,0)
to [out=down, in=up](2,-1)
to [out =down, in =up](2,-2)
to [out =down, in = right](1.5,-2.5);
\strand[thick](2,0)
to [out = down, in =up](1,-1)
to [out =down, in =up](0,-2)
to [out = down,in =up](0,-3.5)
to [out = down, in = left](0.5,-4);
\strand[thick](1.5,-2.5)
to [out = down, in =up](3,-3)
to [out = down, in = up](1,-3.7)
to [out = down, in = right](0.5,-4);
\strand[thick](0.5,-4)
to [out = down, in =up](0.5,-4.5);
\strand[line width=1.8,red](2.5,0)
to [out=down, in=up](2.5,-4.5);
\flipcrossings{4}
\end{knot}
\end{scope}
\node at (3.5,-1.4) {$\underset{\eqref{EqBraid1}}{=}$};
\begin{scope}[shift={(4,-0.2)},scale=0.75]
\begin{knot}[clip width=5, clip radius=8pt]
\strand[thick](0,0)
to [out = down, in =left](0.5,-0.5);
\strand[thick](1,0)
to [out = down, in = right](0.5,-0.5);
\strand[thick](0.5,-0.5)
to [out = down, in =up](1.5,-1.5)
to [out =down, in = up](2.5,-2)
to [out = down, in =up](1.5,-2.5)
to [out = down, in = right](1,-3);
\strand[thick](1.5,0)
to [out=down, in=up](1.5,-0.5)
to [out =down, in =up](0.5,-1.5)
to [out =down, in = up](0.5,-2.5)
to [out = down, in = left](1,-3);
\strand[thick](1,-3)
to [out = down, in = up](1,-3.5);
\strand[line width=1.8,red](2,0)
to [out=down, in=up](2,-3.5);
\flipcrossings{3}
\end{knot}
\end{scope}
\node at (6.7,-1.4) {$= $};
\begin{scope}[shift={(7.2,-0.4)}]
\begin{knot}[clip width=5, clip radius=8pt]
\strand[thick](0,0)
to [out=down, in=left](0.5,-0.5);
\strand[thick](1,0)
to [out = down,in = right](0.5,-0.5);
\strand[thick](0.5,-0.5)
to [out = down, in = left](1,-1);
\strand[line width=1.8,green](1.5,0)
to [out = down, in = right](1,-1);
\strand[line width=1.8,green](1,-1)
to [out=down, in=up](1,-2);
\end{knot}
\end{scope}
\end{tikzpicture}
\end{figure}

The unit property for the action is easily checked, as is property \eqref{EqUnitk}. 

Finally, the action property for $m_Y^r$ follows immediately from the action property for $m_X$ together with \eqref{EqBraid1}.

The fact that $k_Y$ satisfies the reflection equation is easily proven algebraically: writing $k = k_X$, we have
\begin{eqnarray*} 
&& \hspace{1.4cm}(r\times \id_Y)(\id_A\times k_Y)(r\times \id_Y)(\id_A\times k_Y) \\
&=& r_{12}r_{23}k_{34}r_{23}r_{12}r_{23}k_{34}r_{23} \quad  = \quad  r_{12}r_{23}k_{34}r_{12}r_{23}r_{12}k_{34}r_{23}\\
&=& r_{12}r_{23}r_{12}k_{34}r_{23}r_{12}k_{34}r_{23}\quad = \quad  r_{23}r_{12}r_{23}k_{34}r_{23}k_{34}r_{12}r_{23}\\
&=& r_{23}r_{12}k_{34}r_{23}k_{34}r_{23}r_{12}r_{23}\quad = \quad  r_{23}r_{12}k_{34}r_{23}k_{34}r_{12}r_{23}r_{12}\\
&=& r_{23}k_{34}r_{12}r_{23}r_{12}k_{34}r_{23}r_{12}\quad  = \quad  r_{23}k_{34}r_{23}r_{12}r_{23}k_{34}r_{23}r_{12}\\
&&\hspace{1cm} =   (\id_A\times k_Y)(r\times \id_Y)(\id_A\times k_Y)(r\times \id_Y)
\end{eqnarray*}

Let us now check \eqref{EqBraid2}, which is common to all of the above actions as it only involves $k_Y$, $r$ and $m_A$. We compute
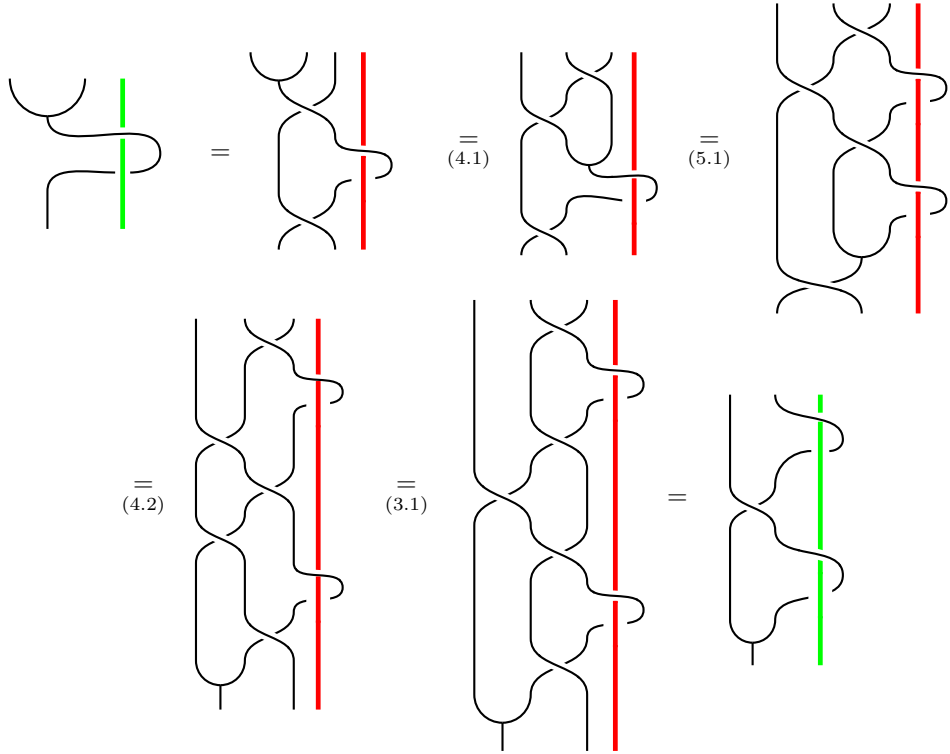
\begin{figure}[h]
\centering
\begin{tikzpicture}
\begin{scope}[shift ={(-1,0)}]
\begin{knot}[clip width=5, clip radius=8pt]
\strand[thick](0,0)
to [out=down, in=left](0.5,-0.5);
\strand[thick](1,0)
to [out = down, in = right](0.5,-0.5);
\strand[thick](0.5,-0.5)
to [out = down,in = up](2,-1)
to [out = down, in = right](1.6,-1.25);
\strand[thick](1.4,-1.25)
to [out = left, in = up](0.5,-1.5)
to [out = down,in = up](0.5,-2);
\strand[line width=1.8,green](1.5,0)
to [out = down, in = up](1.5,-0.65);
\strand[line width=1.8,green](1.5,-0.8)
to [out=down, in=up](1.5,-2);
\end{knot}
\end{scope}
\node at (1.8,-1) {$= $};
\begin{scope}[shift={(2.2,0.35)},scale=0.75]
\begin{knot}[clip width=5, clip radius=8pt]
\strand[thick](0,0)
to [out = down, in =left](0.5,-0.5);
\strand[thick](1,0)
to [out = down, in = right](0.5,-0.5);
\strand[thick](0.5,-0.5)
to [out = down, in = up](1.5,-1.5)
to [out = down, in = up](2.5,-2)
to [out = down, in = up](1.5,-2.5)
to [out = down, in = up](0.5,-3.5);
\strand[thick](1.5,0)
to [out = down, in = up](1.5,-0.5)
to [out = down, in = up](0.5,-1.5)
to [out = down, in = up](0.5,-2.5)
to [out = down, in = up](1.5,-3.5);
\strand[line width = 1.8, red](2,0)
to [out = down, in = up](2,-3.5);
\flipcrossings{2,4}
\end{knot}
\end{scope}
\node at (5.1,-1) {$\underset{\eqref{EqBraid1}}{=}$};
\begin{scope}[shift={(5.8,0.35)},scale=0.6]
\begin{knot}[clip width=5, clip radius=8pt]
\strand[thick](0,0)
to [out = down, in = up](0,-1)
to [out = down, in = up](1,-2)
to [out = down, in = left](1.5,-2.5);
\strand[thick](1,0)
to [out = down, in = up](2,-1)
to [out = down, in = up](2,-2)
to [out = down, in = right](1.5,-2.5);
\strand[thick](1.5,-2.5)
to [out = down, in = up](3,-3)
to [out = down, in = up](1,-3.5)
to [out = down, in = up](0,-4.5);
\strand[thick](2,0)
to [out = down, in = up](1,-1)
to [out = down, in = up](0,-2)
to [out = down, in = up](0,-3.5)
to [out = down, in = up](1,-4.5);
\strand[line width = 1.8, red](2.5,0)
to [out = down, in = up](2.5,-4.5);
\flipcrossings{3,5};
\end{knot}
\end{scope}
\node at (8.3,-1) {$\underset{\eqref{EqBraid2}}{=}$};
\begin{scope}[shift={(9.2,1)},scale=0.75]
\begin{knot}[clip width=5, clip radius=8pt]
\strand[thick](0,0)
to [out = down, in = up](0,-1)
to [out = down, in =up](1,-2)
to [out =down, in = up](2,-3)
to [out = down, in = up](3,-3.5)
to [out = down, in =up](2,-4)
to [out = down, in = right](1.5,-4.5);
\strand[thick](1,0)
to [out = down, in = up](2,-1)
to [out = down, in = up](3,-1.5)
to [out = down, in = up](2,-2)
to [out = down, in =up](1,-3)
to [out = down, in = up](1,-4)
to [out = down, in = left](1.5,-4.5);
\strand[thick](1.5,-4.5)
to [out = down, in =up](0,-5.5);
\strand[thick](2,0)
to [out = down, in =up](1,-1)
to [out = down, in = up](0,-2)
to [out = down, in = up](0,-4.5)
to [out = down, in =up](1.5,-5.5);
\strand[line width = 1.8, red](2.5,0)
to [out = down, in = up](2.5,-5.5);
\flipcrossings{7,4,8};
\end{knot}
\end{scope}
\end{tikzpicture}
\end{figure}

\vspace{-1cm}

\begin{figure}[h]
\centering
\begin{tikzpicture}
\node at (0,-1.4) {$\underset{\eqref{EqBraidr2}}{=}$};
\begin{scope}[shift={(0.7,1)},scale=0.65]
\begin{knot}[clip width=5, clip radius=8pt]
\strand[thick](0,0)
to [out = down, in = up](0,-2)
to [out = down, in =up](1,-3)
to [out = down, in = up](2,-4)
to [out = down, in = up](2,-5)
to [out = down, in = up](3,-5.5)
to [out = down, in = up](2,-6)
to [out = down, in = up](1,-7)
to [out = down, in =right](0.5,-7.5);
\strand[thick](1,0)
to [out = down, in = up](2,-1)
to [out = down, in =up](3,-1.5)
to [out = down, in = up](2,-2)
to [out = down, in = up](2,-3)
to [out = down, in = up](1,-4)
to [out = down, in = up](0,-5)
to [out = down, in = up](0,-7)
to [out = down, in =left](0.5,-7.5);
\strand[thick](0.5,-7.5)
to [out = down, in = up](0.5,-8);
\strand[thick](2,0)
to [out = down, in = up](1,-1)
to [out = down, in =up](1,-2)
to [out = down, in = up](0,-3)
to [out = down, in = up](0,-4)
to [out = down, in = up](1,-5)
to [out = down, in = up](1,-6)
to [out = down, in = up](2,-7)
to [out = down, in =up](2,-8);
\strand[line width = 1.8, red](2.5,0)
to [out = down, in = up](2.5,-8);
\flipcrossings{9,5,3,7};
\end{knot}
\end{scope}
\node at (3.5,-1.4) {$\underset{\eqref{EqBraidFund}}{=}$};
\begin{scope}[shift={(4.4,1.25)},scale=0.75]
\begin{knot}[clip width=5, clip radius=8pt]
\strand[thick](0,0)
to [out = down, in = up](0,-3)
to [out = down, in =up](1,-4)
to [out = down, in = up](2,-5)
to [out = down, in = up](3,-5.5)
to [out = down, in = up](2,-6)
to [out = down, in = up](1,-7)
to [out = down, in =right](0.5,-7.5);
\strand[thick](1,0)
to [out = down, in = up](2,-1)
to [out = down, in =up](3,-1.5)
to [out = down, in = up](2,-2)
to [out = down, in = up](1,-3)
to [out = down, in = up](0,-4)
to [out = down, in = up](0,-5)
to [out = down, in = up](0,-7)
to [out = down, in =left](0.5,-7.5);
\strand[thick](0.5,-7.5)
to [out = down, in = up](0.5,-8);
\strand[thick](2,0)
to [out = down, in = up](1,-1)
to [out = down, in =up](1,-2)
to [out = down, in = up](2,-3)
to [out = down, in = up](2,-4)
to [out = down, in = up](1,-5)
to [out = down, in = up](1,-6)
to [out = down, in = up](2,-7)
to [out = down, in =up](2,-8);
\strand[line width = 1.8, red](2.5,0)
to [out = down, in = up](2.5,-8);
\flipcrossings{9,5,3,7};
\end{knot}
\end{scope}
\node at (7.1,-1.4) {$= $};
\begin{scope}[shift={(7.8,-0.01)},scale = 0.6]
\begin{knot}[clip width=5,clip radius=8pt]
\strand[thick](0,0)
to [out=down, in=up](0,-2)
to [out=down, in=up](1,-3)
to [out =down, in=up](2.5,-4)
to [out =down, in=up](1,-5)
to [out =down, in=right](0.5,-5.5);
\strand[thick](1,0)
to [out=down, in=up](2.5,-1)
to [out = down, in = right](2.2,-1.25);
\strand[thick](1.8,-1.25)
to [out= left, in=up](1,-2)
to [out =down, in = up](0,-3)
to [out = down, in = up](0,-5)
to [out = down, in = left](0.5,-5.5);
\strand[thick](0.5,-5.5)
to [out = down,in = up](0.5,-6);
\strand[line width=1.8,green](2,0)
to [out=down, in=up](2,-0.4);
\strand[line width = 1.8,green](2,-0.6)
to [out=down, in=up](2,-3)
to [out=down, in=up](2,-4)
to [out=down, in=up](2,-5)
to [out=down, in=up](2,-6);
\flipcrossings{3}
\end{knot}
\end{scope}
\end{tikzpicture}
\end{figure}

We now prove \eqref{EqBraidk1} for $m_Y^k$, the other two cases being similar. We have

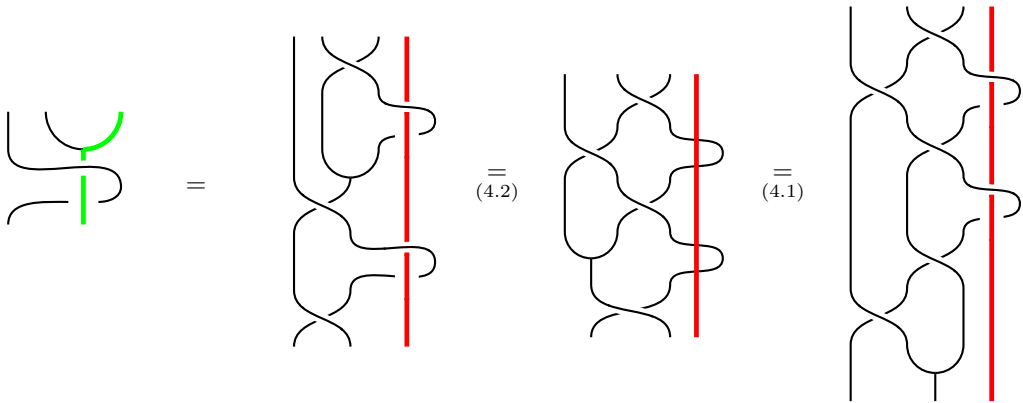
\begin{figure}[h]
\centering
\begin{tikzpicture}
\begin{scope}[shift ={(-2,-1)}]
\begin{knot}[clip width=5, clip radius=8pt]
\strand[thick](0,0)
to [out = down, in = up](0,-0.5)
to [out = down, in = up](1.5,-1)
to [out = down, in = right](1.2,-1.2);
\strand[thick](0.8,-1.2)
to [out = left, in = up](0,-1.5);
\strand[thick](0.5,0)
to [out = down, in = left](1,-0.5);
\strand[line width = 1.8, green](1.5,0)
to [out = down, in = right](1,-0.5);
\strand[line width = 1.8,green](1,-0.5)
to [out = down, in = up](1,-0.65);
\strand[line width = 1.8,green](1,-0.85)
to [out = down, in = up](1,-1.5);
\end{knot}
\end{scope}
\node at (0.5,-2) {$= $};
\begin{scope}[shift={(1.8,0)},scale=0.75]
\begin{knot}[clip width=5, clip radius=8pt]
\strand[thick](0,0)
to [out = down, in = up](0,-2.5)
to [out = down, in = up](1,-3.5)
to [out = down, in = up](2.5,-4)
to [out = down, in = up](1,-4.5)
to [out = down, in = up](0,-5.5);
\strand[thick](0.5,0)
to [out = down, in = up](1.5,-1)
to [out = down, in = up](2.5,-1.5)
to [out = down, in = up](1.5,-2)
to [out = down, in = right](1,-2.5);
\strand[thick](1.5,0)
to [out = down, in = up](0.5,-1)
to [out = down, in =up](0.5,-2)
to[out = down, in = left](1,-2.5);
\strand[thick](1,-2.5)
to [out = down, in = up](0,-3.5)
to [out = down, in  =up](0,-4.5)
to [out = down, in = up](1,-5.5);
\strand[line width = 1.8,red](2,0)
to [out = down, in = up](2,-5.5);
\flipcrossings{7,2,4};
\end{knot}
\end{scope}
\node at (4.5,-2) {$\underset{\eqref{EqBraidr2}}{=}$};
\begin{scope}[shift={(5.4,-0.5)},scale=0.7]
\begin{knot}[clip width=5, clip radius=8pt]
\strand[thick](0,0)
to [out = down, in = up](0,-1)
to [out = down, in = up](1,-2)
to [out = down, in = up](2,-3)
to [out = down, in = up](3,-3.5)
to [out = down, in = up](2,-4)
to [out = down, in = up](0.5,-5);
\strand[thick](1,0)
to [out = down, in = up](2,-1)
to [out = down, in = up](3,-1.5)
to [out = down, in = up](2,-2)
to [out = down, in = up](1,-3)
to [out = down, in = right](0.5,-3.5);
\strand[thick](2,0)
to [out = down, in = up](1,-1)
to [out = down, in = up](0,-2)
to [out = down, in = up](0,-3)
to [out = down, in =left](0.5,-3.5);
\strand[thick](0.5,-3.5)
to [out = down, in = up](0.5,-4)
to [out = down, in = up](2,-5);
\strand[line width = 1.8,red](2.5,0)
to [out = down, in = up](2.5,-5);
\flipcrossings{3,5,8};
\end{knot}
\end{scope}
\node at (8.3,-2) {$\underset{\eqref{EqBraid1}}{=}$};
\begin{scope}[shift={(9.2,0.4)},scale=0.75]
\begin{knot}[clip width=5, clip radius=8pt]
\strand[thick](0,0)
to [out = down, in = up](0,-1)
to [out = down, in = up](1,-2)
to [out = down, in = up](2,-3)
to [out = down, in = up](3,-3.5)
to [out = down, in = up](2,-4)
to [out = down, in = up](1,-5)
to [out = down, in = up](0,-6)
to [out = down, in = up](0,-7);
\strand[thick](1,0)
to [out = down, in = up](2,-1)
to [out = down, in = up](3,-1.5)
to [out = down, in = up](2,-2)
to [out = down, in = up](1,-3)
to [out = down, in = up](1,-4)
to [out = down, in =up](2,-5)
to [out = down, in = up](2,-6)
to [out = down, in = right](1.5,-6.5);
\strand[thick](2,0)
to [out = down, in = up](1,-1)
to [out = down, in = up](0,-2)
to [out = down, in = up](0,-5)
to [out = down, in = up](1,-6)
to [out = down, in =left](1.5,-6.5);
\strand[thick](1.5,-6.5)
to [out = down, in = up](1.5,-7);
\strand[line width = 1.8,red](2.5,0)
to [out = down, in = up](2.5,-7);
\flipcrossings{9,6,4,2};
\end{knot}
\end{scope}
\end{tikzpicture}
\end{figure}

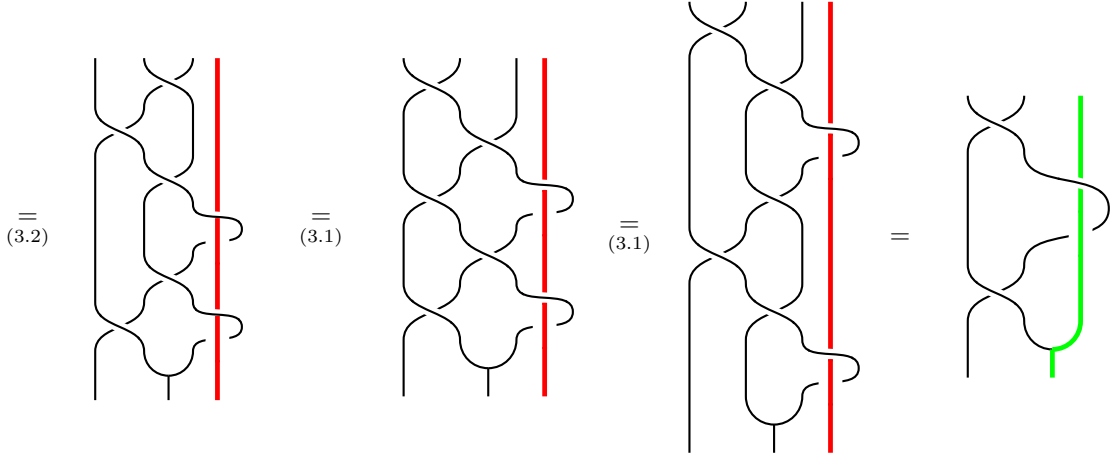
\begin{figure}[h]
\centering
\begin{tikzpicture}
\node at (-1,-1.3) {$\underset{\eqref{EqRE}}{=}$};
\begin{scope}[shift={(-0.1,1)},scale=0.65]
\begin{knot}[clip width=5, clip radius=8pt]
\strand[thick](0,0)
to [out = down, in = up](0,-1)
to [out = down, in = up](1,-2)
to [out = down, in = up](2,-3)
to [out = down, in = up](3,-3.5)
to [out = down, in = up](2,-4)
to [out = down, in = up](1,-5)
to [out = down, in = up](0,-6)
to [out = down, in = up](0,-7);
\strand[thick](1,0)
to [out = down, in = up](2,-1)
to [out = down, in = up](2,-2)
to [out = down, in = up](1,-3)
to [out = down, in = up](1,-4)
to [out = down, in =up](2,-5)
to [out = down, in = up](3,-5.5)
to [out = down, in = up](2,-6)
to [out = down, in = right](1.5,-6.5);
\strand[thick](2,0)
to [out = down, in = up](1,-1)
to [out = down, in = up](0,-2)
to [out = down, in = up](0,-5)
to [out = down, in = up](1,-6)
to [out = down, in =left](1.5,-6.5);
\strand[thick](1.5,-6.5)
to [out = down, in = up](1.5,-7);
\strand[line width = 1.8,red](2.5,0)
to [out = down, in = up](2.5,-7);
\flipcrossings{9,6,4,2};
\end{knot}
\end{scope}
\node at (2.9,-1.3) {$\underset{\eqref{EqBraidFund}}{=}$};
\begin{scope}[shift={(4,1)},scale=0.75]
\begin{knot}[clip width=5, clip radius=8pt]
\strand[thick](0,0)
to [out = down, in = up](1,-1)
to [out = down, in = up](2,-2)
to [out = down, in = up](3,-2.5)
to [out = down, in = up](2,-3)
to [out = down, in = up](1,-4)
to [out = down, in = up](0,-5)
to [out = down, in = up](0,-6);
\strand[thick](1,0)
to [out = down, in = up](0,-1)
to [out = down, in = up](0,-2)
to [out = down, in = up](1,-3)
to [out = down, in =up](2,-4)
to [out = down, in = up](3,-4.5)
to [out = down, in = up](2,-5)
to [out = down, in = right](1.5,-5.5);
\strand[thick](2,0)
to [out = down, in =up](2,-1)
to [out = down, in = up](1,-2)
to [out = down, in = up](0,-3)
to [out = down, in = up](0,-4)
to [out = down, in = up](1,-5)
to [out = down, in =left](1.5,-5.5);
\strand[thick](1.5,-5.5)
to [out = down, in = up](1.5,-6);
\strand[line width = 1.8,red](2.5,0)
to [out = down, in = up](2.5,-6);
\flipcrossings{9,6,4,2};
\end{knot}
\end{scope}
\node at (7,-1.4) {$\underset{\eqref{EqBraidFund}}{=}$};
\begin{scope}[shift={(7.8,1.75)},scale = 0.75]
\begin{knot}[clip width=5, clip radius=8pt]
\strand[thick](0,0)
to [out = down, in = up](1,-1)
to [out = down, in = up](2,-2)
to [out = down, in = up](3,-2.5)
to [out = down, in = up](2,-3)
to [out = down, in = up](1,-4)
to [out = down, in = up](0,-5)
to [out = down, in = up](0,-8);
\strand[thick](1,0)
to [out = down, in = up](0,-1)
to [out = down, in = up](0,-4)
to [out = down, in = up](1,-5)
to [out = down, in =up](2,-6)
to [out = down, in = up](3,-6.5)
to [out = down, in = up](2,-7)
to [out = down, in = right](1.5,-7.5);
\strand[thick](2,0)
to [out = down, in =up](2,-1)
to [out = down, in = up](1,-2)
to [out = down, in = up](1,-3)
to [out = down, in = up](2,-4)
to [out = down, in = up](2,-5)
to [out = down, in = up](1,-6)
to [out = down, in =up](1,-7)
to [out = down, in =left](1.5,-7.5);
\strand[thick](1.5,-7.5)
to [out = down, in = up](1.5,-8);
\strand[line width = 1.8,red](2.5,0)
to [out = down, in = up](2.5,-8);
\flipcrossings{9,6,4,2};
\end{knot}
\end{scope}
\node at (10.6,-1.4) {$= $};
\begin{scope}[shift={(11.5,0.5)},scale = 0.75]
\begin{knot}[clip width=5,clip radius=8pt]
\strand[thick](0,0)
to [out=down, in=up](1,-1)
to [out=down, in=up](2.5,-2)
to [out =down, in=up](1,-3)
to [out =down, in=up](0,-4)
to [out =down, in=up](0,-5);
\strand[thick](1,0)
to [out=down, in=up](0,-1)
to [out=down, in=up](0,-2)
to [out =down, in = up](0,-3)
to [out = down, in = up](1,-4)
to [out = down, in = left](1.5,-4.5);
\strand[line width=1.8pt,green](2,0)
to [out=down, in=up](2,-4)
to [out=down, in =right](1.5,-4.5);
\strand[line width=1.8pt,green](1.5,-4.5)
to [out = down, in =  up](1.5,-5);
\flipcrossings{2,4,6}
\end{knot}
\end{scope}
\end{tikzpicture}
\end{figure}
\end{proof}

It follows in particular that, starting from a braided action $(X,m_X,k_X)$ of $(A,m,r)$, we obtain solutions to the reflection equation over the sets $A^{\times n}\times X$.

%We will now characterize braided actions $(X,m_X,k)$ in terms of a pair of compatible actions of $(A,\circ)$ on the pair of sets $X$ and $A\times X$. We use the following terminology. 

\begin{Def}
We say two actions $m_Y,m_Y'$ of $(A,\circ)$ on a set $Y$ \emph{braid-commute} if 
\begin{equation}\label{EqCommBraid1}
m_Y(\id_A \times m_Y') = m_Y'(\id_A\times m_Y)(r\times \id_Y). 
\end{equation}
\end{Def}
%What about transitive?
Note that the relation of braid-commuting is not symmetric if $r$ is not involutive!

\begin{figure}[h]
\centering
\begin{tikzpicture}
\node at (-2.5,-1) {$m_Y = $};
\begin{scope}[shift={(-1.5,0)}]
\begin{knot}[clip width=5, clip radius=8pt]
\strand[thick](0,0)
to [out=down, in=left](1,-1);
\strand[line width=1.8,green](2,0)
to [out = down, in = right](1,-1);
\strand[line width=1.8,green](1,-1)
to [out=down, in=up](1,-2);
\end{knot}
\end{scope}

\node at (2,-1) {$m_Y' = $};
\begin{scope}[shift={(3,0)}]
\begin{knot}[clip width=5, clip radius=8pt]
\strand[thick,style = dashed](0,0)
to [out=down, in=left](1,-1);
\strand[line width=1.8,green](2,0)
to [out = down, in = right](1,-1);
\strand[line width=1.8,green](1,-1)
to [out=down, in=up](1,-2);
\end{knot}
\end{scope}

\begin{scope}[shift = {(6.5,0)}]
\begin{knot}[clip width=5, clip radius=8pt]
\strand[thick](0,0)
to [out=down, in=left](0.5,-1);
\strand[thick, style = {dashed}](0.5,0)
to [out = down,in = left](1,-0.5);
\strand[line width = 1.8,green](1.5,0)
to [out = down, in = right](1,-0.5);
\strand[line width=1.8,green](1,-0.5)
to [out = down, in = right](0.5,-1);
\strand[line width=1.8,green](0.5,-1)
to [out=down, in=up](0.5,-2);
\end{knot}
\end{scope}
\node at (8.3,-1) {$= $};
\begin{scope}[shift = {(9.1,0.2)},scale  =0.8]
\begin{knot}[clip width=5, clip radius=8pt]
\strand[thick,style](0.5,0)
to [out = down, in = up](0,-1);
\strand[thick, style = dashed](0,-1)
to [out=down, in=left](0.5,-2);
\strand[thick](0,0)
to [out = down, in = up](0.5,-1)
to [out = down,in = left](1,-1.5);
\strand[line width = 1.8,green](1.5,0)
to [out =down, in =up](1.5,-1)
to [out = down, in = right](1,-1.5);
\strand[line width=1.8,green](1,-1.5)
to [out = down, in = right](0.5,-2);
\strand[line width=1.8,green](0.5,-2)
to [out=down, in=up](0.5,-3);
\flipcrossings{1}
\end{knot}
\end{scope}
\end{tikzpicture}
\caption{Braid-commuting actions}
\label{fig:BraidCommutingActions}
\end{figure}
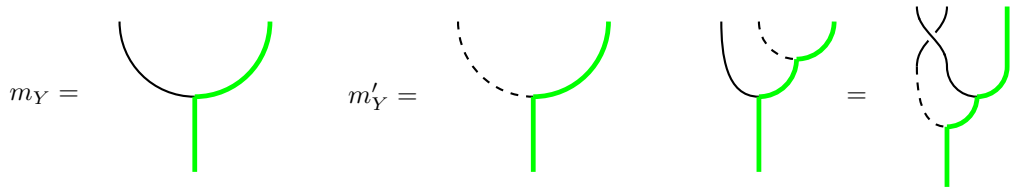

\begin{Prop}\label{PropBraidComm} 
Let $(X,m_X,k)$ be a generalized braided action. Then the pairs $(m_Y^k,m_Y^{\triv}), (m_Y^k,m_Y^r)$ and $(m_Y^r,m_Y^{\triv})$ braid-commute.
\end{Prop}
\begin{proof}
We leave the braid-commutation of $m_Y^{r}$ and $m_Y^{\triv}$ to the reader.

The braid-commutation of $m_Y^{k}$ and $m_Y^{\triv}$ follows from 

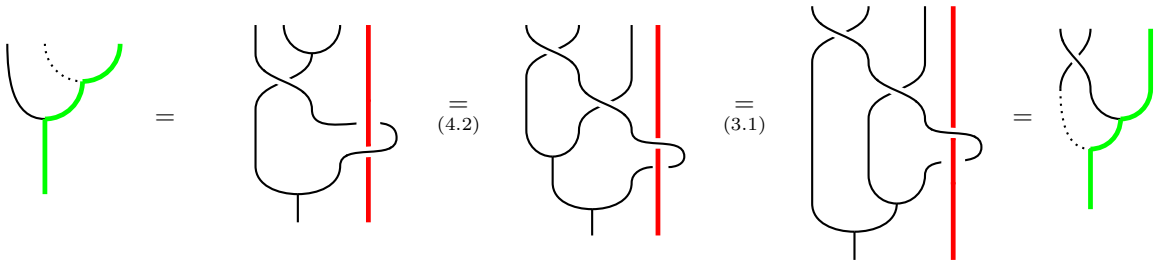
\begin{figure}[h]
\centering
\begin{tikzpicture}
\begin{scope}[shift = {(-1.5,0)}]
\begin{knot}[clip width=5, clip radius=8pt]
\strand[thick](0,0)
to [out=down, in=left](0.5,-1);
\strand[thick, style = {dotted}](0.5,0)
to [out = down,in = left](1,-0.5);
\strand[line width = 1.8,green](1.5,0)
to [out = down, in = right](1,-0.5);
\strand[line width=1.8,green](1,-0.5)
to [out = down, in = right](0.5,-1);
\strand[line width=1.8,green](0.5,-1)
to [out=down, in=up](0.5,-2);
\end{knot}
\end{scope}
\node at (0.6,-1) {$= $};
\begin{scope}[shift={(1.8,0.25)},scale=0.75]
\begin{knot}[clip width=5, clip radius=8pt]
\strand[thick](0,0)
to [out = down, in =up](0,-0.5)
to [out =down, in = up](1,-1.5)
to [out = down, in = up](2.5,-2)
to [out = down, in =up](1.5,-2.5)
to [out = down, in = right](0.75,-3);
\strand[thick](0.5,0)
to [out =down, in =left](1,-0.5);
\strand[thick](1.5,0)
to [out = down, in = right](1,-0.5);
\strand[thick](1,-0.5)
to [out =down, in =up](0,-1.5)
to [out = down, in = up](0,-2.5)
to [out = down, in = left](0.75,-3);
\strand[thick](0.75,-3)
to [out = down, in = up](0.75,-3.5);
\strand[line width = 1.8,red](2,0)
to [out =down, in =up](2,-3.5);
\flipcrossings{2}
\end{knot}
\end{scope}
\node at (4.5,-1) {$\underset{\eqref{EqBraidr2}}{=}$};
\begin{scope}[shift={(5.4,0.25)},scale=0.7]
\begin{knot}[clip width=5, clip radius=8pt]
\strand[thick](0,0)
to [out = down, in = up](1,-1)
to [out = down, in = up](2,-2)
to [out = down, in =up](3,-2.5)
to [out =down, in =right](2.7,-2.7);
\strand[thick](2.4,-2.7)
to [out = left, in = up](2,-3)
to [out = down, in =right](1.25,-3.5);
\strand[thick](1,0)
to [out = down, in = up](0,-1)
to [out = down, in = up](0,-2)
to [out = down, in = left](0.5,-2.5);
\strand[thick](2,0)
to [out = down, in = up](2,-1)
to[out = down, in = up](1,-2)
to[out = down, in =right](0.5,-2.5);
\strand[thick](0.5,-2.5)
to [out = down, in =up](0.5,-3)
to [out =down, in = left](1.25,-3.5);
\strand[thick](1.25,-3.5)
to [out =down, in =up](1.25,-4);
\strand[line width = 1.8,red](2.5,0)
to [out =down, in = up](2.5,-2.15);
\strand[line width = 1.8,red](2.5,-2.35)
to [out =down, in =up](2.5,-4);
\end{knot}
\end{scope}
\node at (8.3,-1) {$\underset{\eqref{EqBraidFund}}{=}$};
\begin{scope}[shift={(9.2,0.5)},scale=0.75]
\begin{knot}[clip width=5, clip radius=8pt]
\strand[thick](0,0)
to [out = down, in = up](1,-1)
to [out = down, in = up](2,-2)
to [out = down, in =up](3,-2.5)
to [out = down, in = up](2,-3)
to [out = down, in =right](1.5,-3.5);
\strand[thick](1,0)
to [out = down, in = up](0,-1)
to [out = down, in = up](0,-3.5)
to [out = down, in = left](0.75,-4);
\strand[thick](2,0)
to [out = down, in = up](2,-1)
to[out = down, in = up](1,-2)
to[out = down, in =up](1,-3)
to[out = down, in = left](1.5,-3.5);
\strand[thick](1.5,-3.5)
to [out = down, in =right](0.75,-4);
\strand[thick](0.75,-4)
to [out =down, in =up](0.75,-4.5);
\strand[line width = 1.8,red](2.5,0)
to [out =down, in =up](2.5,-4.5);
\flipcrossings{4};
\end{knot}
\end{scope}
\node at (12,-1) {$=$};
\begin{scope}[shift = {(12.5,0.2)},scale  =0.8]
\begin{knot}[clip width=5, clip radius=8pt]
\strand[thick,style](0.5,0)
to [out = down, in = up](0,-1);
\strand[thick, style = dotted](0,-1)
to [out=down, in=left](0.5,-2);
\strand[thick](0,0)
to [out = down, in = up](0.5,-1)
to [out = down,in = left](1,-1.5);
\strand[line width = 1.8,green](1.5,0)
to [out =down, in =up](1.5,-1)
to [out = down, in = right](1,-1.5);
\strand[line width=1.8,green](1,-1.5)
to [out = down, in = right](0.5,-2);
\strand[line width=1.8,green](0.5,-2)
to [out=down, in=up](0.5,-3);
\flipcrossings{1};
\end{knot}
\end{scope}
\end{tikzpicture}
\end{figure}

The braid-commutation of $m_Y^{k}$ and $m_Y^r$ is proven as follows: 
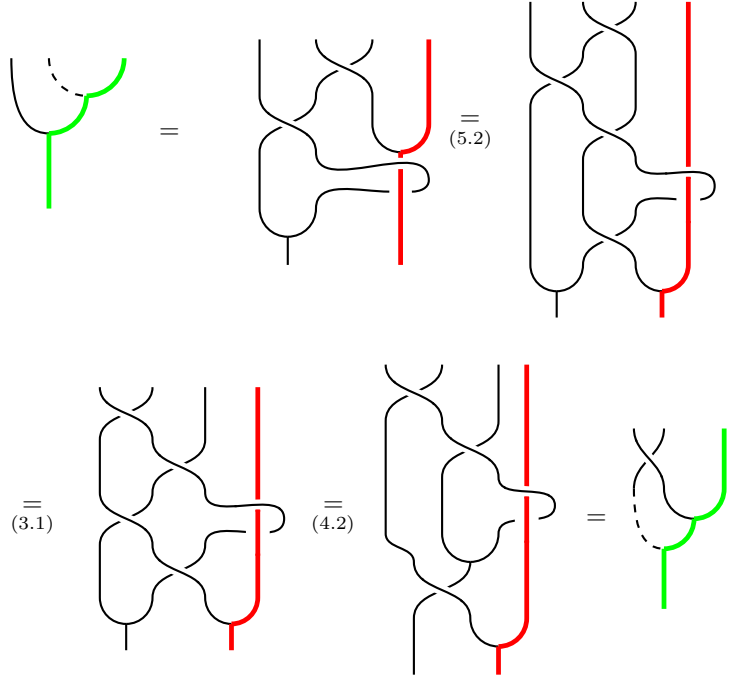
\begin{figure}[h]
\centering
\begin{tikzpicture}
\begin{scope}[shift = {(-1.5,0)}]
\begin{knot}[clip width=5, clip radius=8pt]
\strand[thick](0,0)
to [out=down, in=left](0.5,-1);
\strand[thick, style = {dashed}](0.5,0)
to [out = down,in = left](1,-0.5);
\strand[line width = 1.8,green](1.5,0)
to [out = down, in = right](1,-0.5);
\strand[line width=1.8,green](1,-0.5)
to [out = down, in = right](0.5,-1);
\strand[line width=1.8,green](0.5,-1)
to [out=down, in=up](0.5,-2);
\end{knot}
\end{scope}
\node at (0.6,-1) {$= $};
\begin{scope}[shift={(1.8,0.25)},scale=0.75]
\begin{knot}[clip width=5, clip radius=8pt]
\strand[thick](0,0)
to [out = down, in =up](0,-1)
to [out =down, in = up](1,-2)
to [out = down, in = up](3,-2.5)
to [out = down, in = right](2.7,-2.7);
\strand[thick](2.3,-2.7)
to [out = left, in = up](1,-3)
to [out = down, in = right](0.5,-3.5);
\strand[thick](1,0)
to [out = down, in = up](2,-1)
to [out = down, in = up](2,-1.5)
to[out = down, in = left](2.5,-2);
\strand[thick](2,0)
to [out =down, in = up](1,-1)
to [out = down, in = up](0,-2)
to [out = down, in = up](0,-3)
to[out = down, in = left](0.5,-3.5);
\strand[thick](0.5,-3.5)
to [out = down, in = up](0.5,-4);
\strand[line width = 1.8,red](3,0)
to [out =down, in =up](3,-1.5)
to [out = down, in = right](2.5,-2);
\strand[line width = 1.8,red](2.5,-2)
to [out = down, in =up](2.5,-2.1);
\strand[line width = 1.8,red](2.5,-2.3)
to [out = down, in = up](2.5,-4);
\end{knot}
\end{scope}
\node at (4.6,-1) {$\underset{\eqref{EqBraidk1}}{=}$};
\begin{scope}[shift={(5.4,0.75)},scale=0.7]
\begin{knot}[clip width=5, clip radius=8pt]
\strand[thick](0,0)
to [out =down, in = up](0,-1)
to [out = down, in = up](1,-2)
to [out = down, in = up](2,-3)
to[out = down, in =up](3.5,-3.5)
to [out =down, in = up](2,-4)
to [out = down, in =up](1,-5)
to [out = down, in = right](0.5,-5.5);
\strand[thick](1,0)
to[out = down, in  =up](2,-1)
to [out = down, in = up](2,-2)
to [out = down, in =up](1,-3)
to[out = down, in = up](1,-4)
to[out = down, in = up](2,-5)
to[out = down, in = left](2.5,-5.5);
\strand[thick](2,0)
to[out = down, in =up](1,-1)
to[out = down, in = up](0,-2)
to[out = down, in =up](0,-5)
to[out = down, in =left](0.5,-5.5);
\strand[thick](0.5,-5.5)
to [out = down, in =up](0.5,-6);
\strand[line width = 1.8,red](3,0)
to [out = down, in = up] (3,-5)
to [out = down, in = right](2.5,-5.5);
\strand[line width = 1.8,red](2.5,-5.5)
to [out =down, in = up](2.5,-6);
\flipcrossings{2,5};
\end{knot}
\end{scope}
\end{tikzpicture}
\end{figure}

\newpage

\begin{figure}[h]
\centering
\begin{tikzpicture}
\node at (4.5,-1) {$\underset{\eqref{EqBraidFund}}{=}$};
\begin{scope}[shift={(5.4,0.75)},scale=0.7]
\begin{knot}[clip width=5, clip radius=8pt]
\strand[thick](0,0)
to [out =down, in = up](1,-1)
to [out = down, in = up](2,-2)
to[out = down, in =up](3.5,-2.5)
to [out =down, in = up](2,-3)
to [out = down, in =up](1,-4)
to [out = down, in = right](0.5,-4.5);
\strand[thick](1,0)
to[out = down, in  =up](0,-1)
to [out = down, in = up](0,-2)
to [out = down, in =up](1,-3)
to[out = down, in = up](2,-4)
to[out = down, in = left](2.5,-4.5);
\strand[thick](2,0)
to[out = down, in =up](2,-1)
to[out = down, in = up](1,-2)
to [out = down, in =up](0,-3)
to[out = down, in =up](0,-4)
to[out = down, in =left](0.5,-4.5);
\strand[thick](0.5,-4.5)
to [out = down, in =up](0.5,-5);
\strand[line width = 1.8,red](3,0)
to [out = down, in = up] (3,-4)
to [out = down, in = right](2.5,-4.5);
\strand[line width = 1.8,red](2.5,-4.5)
to [out =down, in = up](2.5,-5);
\flipcrossings{2,5};
\end{knot}
\end{scope}
\node at (8.5,-1) {$\underset{\eqref{EqBraidr2}}{=}$};
\begin{scope}[shift={(9.2,1.05)},scale=0.75]
\begin{knot}[clip width=5, clip radius=8pt]
\strand[thick](0,0)
to [out = down, in = up](1,-1)
to [out = down, in = up](2,-2)
to [out = down, in =up](3,-2.5)
to [out = down, in = up](2,-3)
to [out = down, in =right](1.5,-3.5);
\strand[thick](1,0)
to [out = down, in = up](0,-1)
to [out = down, in = up](0,-3)
to [out = down, in = up](0.5,-3.5)
to [out = down, in =up](1.5,-4.5)
to [out = down, in = left](2,-5);
\strand[thick](2,0)
to [out = down, in = up](2,-1)
to[out = down, in = up](1,-2)
to[out = down, in =up](1,-3)
to[out = down, in = left](1.5,-3.5);
\strand[thick](1.5,-3.5)
to [out = down, in =up](0.5,-4.5)
to [out = down, in =up](0.5,-5.5);
\strand[line width = 1.8,red](2.5,0)
to [out =down, in =up](2.5,-4.5)
to [out = down, in =right](2,-5);
\strand[line width = 1.8,red](2,-5)
to [out = down, in = up](2,-5.5);
\flipcrossings{4};
\end{knot}
\end{scope}
\node at (12,-1) {$= $};
\begin{scope}[shift = {(12.5,0.2)},scale  =0.8]
\begin{knot}[clip width=5, clip radius=8pt]
\strand[thick,style](0.5,0)
to [out = down, in = up](0,-1);
\strand[thick, style = dashed](0,-1)
to [out=down, in=left](0.5,-2);
\strand[thick](0,0)
to [out = down, in = up](0.5,-1)
to [out = down,in = left](1,-1.5);
\strand[line width = 1.8,green](1.5,0)
to [out =down, in =up](1.5,-1)
to [out = down, in = right](1,-1.5);
\strand[line width=1.8,green](1,-1.5)
to [out = down, in = right](0.5,-2);
\strand[line width=1.8,green](0.5,-2)
to [out=down, in=up](0.5,-3);
\flipcrossings{1};
\end{knot}
\end{scope}
\end{tikzpicture}
\end{figure}
\end{proof}

Note that the $(A,\circ)$-actions $m_Y^{\triv}$ and $m_Y^r$ can be constructed purely from the $(A,\circ)$-action $(X,m_X)$.

\begin{Lem}
There is a one-to-one correspondence between 
\begin{itemize}
\item generalized braided actions $(X,m_X,k)$ of $(A,\circ,r)$, and
\item pairs of actions $m_X,m_Y$ of $(A,m_A)$ on the respective sets $X$ and $Y = A\times X$ such that  the pairs $(m_Y,m_Y^{\triv})$ and $(m_Y,m_Y^{r})$ braid-commute. 
\end{itemize}
Moreover, the action is braided if and only if the map 
\[
m_X: Y \rightarrow X
\]
is $(A,\circ)$-equivariant.
\end{Lem}

\begin{proof}
If $(X,m_X,k)$ is a generalized braided action, we have shown in Proposition \ref{PropBraidComm} that the pairs $(m_Y^k,m_Y^{\triv})$ and $(m_Y^k,m_Y^{r})$ braid-commute. If moreover $(X,m_X,k)$ is a braided action, it follows straightforwardly from \eqref{EqBraidCom} and \eqref{EqTrivmXk}, together with the action property of $m_X$, that $m_X$ is an equivariant map from $(Y,m_X^{k})$ to $(X,m_X)$.

Conversely, assume that $(X,\circ)$ is an action of $(A,\circ)$, and that we have an action $m_Y$ of $A$ on $A\times X$, 
\[
m_Y(a,(b,x)) = a*(b,x),
\] 
satisfying the braid-commutation \eqref{EqCommBraid1} with $m_Y^{\triv}$ and $m_Y^r$. Define
\[
k: A\times X \rightarrow A\times X,\quad (a,x) \mapsto k(a,x) := a *(e_A,x).
\]
Let us write 
\[
a\circ (b,x) = (a\circ b,x)
\]
for the action $m_Y^{\triv}$. Take $a,b\in A$ and write $(b',a') = r(a,b)$. Then from the braid-commutation of $m_Y$ with $m_Y^{\triv}$, we obtain that
\[
a*(b,x) = a*(b\circ (e_A,x)) = b' \circ (a'*(e_A,x)) = b'\circ k(a',x),
\]
so that $m_Y$ is determined by the formula \eqref{EqDefMu} with $k_X$ replaced by $k$. We may hence write $k$ and $m_Y$ diagramatically as in resp. Figure \ref{fig:BraidActionStructure} and Figure \ref{fig:ProdBraidStructure}, and are to show that the associated $k$-map satisfies \eqref{EqBraid2}, \eqref{EqBraidk1}, \eqref{EqUnitk}. 

In fact, \eqref{EqBraid2} is immediate upon applying the action equation $m_Y(m_A\times \id_Y) = m_Y(\id_A\times m_Y)$ to $(a,b,e_A,x)$, and similarly \eqref{EqUnitk} follows from $e_A*(e_A,x) = (e_A,x)$. The relation \eqref{EqBraidk1} follows from the following computation, where in the penultimate step we use that $m_Y$ and $m_Y^r$ braid-commute:

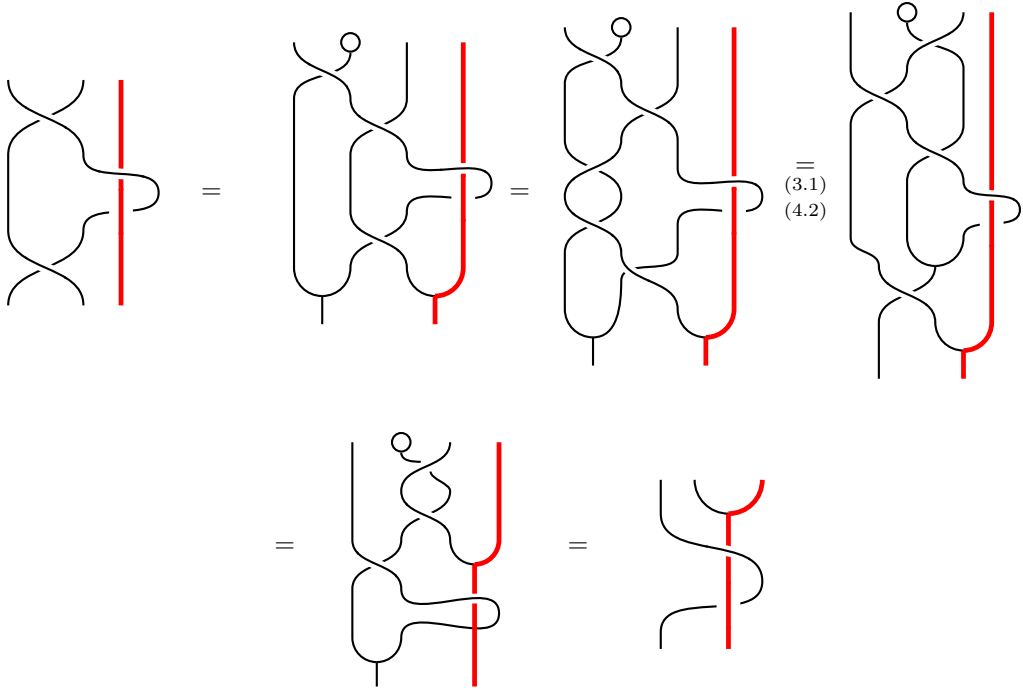
\begin{figure}[h]
\centering
\begin{tikzpicture}
\begin{scope}[shift ={(-2,-0.5)}]
\begin{knot}[clip width=5, clip radius=8pt]
\strand[thick](0,0)
to [out = down, in = up](1,-1)
to [out = down, in = up](2,-1.5)
to [out = down, in = up](1,-2)
to [out = down, in = up](0,-3);
\strand[thick](1,0)
to [out = down, in = up](0,-1)
to[out = down, in = up](0,-2)
to [out = down, in = up](1,-3);
\strand[line width = 1.8, red](1.5,0)
to [out = down, in = up](1.5,-3);
\flipcrossings{2,4}
\end{knot}
\end{scope}
\node at (0.7,-2) {$= $};
\begin{scope}[shift={(1.8,0)},scale=0.75]
\begin{knot}[clip width=5,clip radius=8pt]
\node(A)at(1,0)[circle, draw=black, inner sep=0pt, minimum size=7pt, thick]{};
\strand[thick](0,0)
to [out= down, in = up](1,-1)
to[out = down, in = up](2,-2)
to [out= down, in = up](3.5,-2.5)
to [out = down, in =up](2,-3)
to [out = down, in = up](1,-4)
to [out = down, in =right](0.5,-4.5);
\strand[thick](A.south)
to [out=down, in=up](0,-1)
to [out = down, in =up](0,-4)
to [out = down, in = left](0.5,-4.5);
\strand[thick](0.5,-4.5)
to [out= down, in =up](0.5,-5);
\strand[thick](2,0)
to [out = down, in = up](2,-1)
to [out = down, in =up](1,-2)
to[out = down, in = up](1,-3)
to [out = down, in = up](2,-4)
to [out = down, in = left](2.5,-4.5);
\strand[line width =1.8,red](3,0)
to [out = down, in= up](3,-4)
to [out = down, in =right](2.5,-4.5);
\strand[line width = 1.8, red](2.5,-4.5)
to [out =down, in = up](2.5,-5);
\flipcrossings{3,5};
\end{knot}
\end{scope}
\node at (4.8,-2) {$= $};
\begin{scope}[shift={(5.4,0.2)},scale=0.75]
\begin{knot}[clip width=5,clip radius=8pt]
\node(A)at(1,0)[circle, draw=black, inner sep=0pt, minimum size=7pt, thick]{};
\strand[thick](0,0)
to [out= down, in = up](1,-1)
to[out = down, in = up](2,-2)
to [out = down, in = up](2,-2.5)
to [out= down, in = up](3.5,-3)
to [out = down, in =up](2,-3.5)
to [out =down, in =up](2,-4)
to [out = down, in = up](1,-4.5)
to [out = down, in =right](0.5,-5.5);
\strand[thick](A.south)
to [out=down, in=up](0,-1)
to [out = down, in = up](0,-2)
to [out = down, in =up](1,-3)
to [out = down, in = up](0,-4)
to [out = down, in =up](0,-5)
to [out = down, in = left](0.5,-5.5);
\strand[thick](0.5,-5.5)
to [out= down, in =up](0.5,-6);
\strand[thick](2,0)
to [out = down, in = up](2,-1)
to [out = down, in =up](1,-2)
to [out = down, in = up](0,-3)
to[out = down, in = up](1,-4)
to [out = down, in = up](2,-5)
to [out = down, in = left](2.5,-5.5);
\strand[line width =1.8,red](3,0)
to [out = down, in= up](3,-5)
to [out = down, in =right](2.5,-5.5);
\strand[line width = 1.8, red](2.5,-5.5)
to [out =down, in = up](2.5,-6);
\flipcrossings{3,5,6,7};
\end{knot}
\end{scope}
\node at (8.6,-2) {$\underset{\eqref{EqBraidr2}}{\underset{\eqref{EqBraidFund}}{=}}$};
\begin{scope}[shift={(9.2,0.4)},scale=0.75]
\begin{knot}[clip width=5,clip radius=8pt]
\node(A)at(1,0)[circle, draw=black, inner sep=0pt, minimum size=7pt, thick]{};
\strand[thick](0,0)
to [out= down, in = up](0,-1)
to[out = down, in = up](1,-2)
to[out = down, in = up](2,-3)
to [out= down, in = up](3,-3.5)
to [out = down, in =up](2,-4)
to [out = down, in =right](1.5,-4.5);
\strand[thick](A.south)
to [out=down, in=up](2,-1)
to [out = down, in = up](2,-2)
to [out = down, in =up](1,-3)
to [out = down, in =up](1,-4)
to [out = down, in =left](1.5,-4.5);
\strand[thick](1.5,-4.5)
to [out= down, in =up](0.5,-5.5)
to [out = down, in = up](0.5,-6.5);
\strand[thick](2,0)
to [out = down, in = up](1,-1)
to [out = down, in =up](0,-2)
to [out = down, in = up](0,-4)
to[out = down, in = up](0.5,-4.5)
to [out = down, in = up](1.5,-5.5)
to [out = down, in = left](2,-6);
\strand[line width =1.8,red](2.5,0)
to [out = down, in= up](2.5,-5.5)
to [out = down, in =right](2,-6);
\strand[line width = 1.8, red](2,-6)
to [out =down, in = up](2,-6.5);
\flipcrossings{4,5,6,7};
\end{knot}
\end{scope}
\end{tikzpicture}
\end{figure}
\begin{figure}[h]
\centering
\begin{tikzpicture}
\node at (-1,-0.4) {$= $};
\begin{scope}[shift={(-0.1,1)},scale=0.65]
\begin{knot}[clip width=5,clip radius=8pt]
\node(A)at(1,0)[circle, draw=black, inner sep=0pt, minimum size=7pt, thick]{};
\strand[thick](0,0)
to [out= down, in = up](0,-2)
to[out = down, in = up](1,-3)
to[out = down, in = up](3,-3.5)
to [out= down, in = up](1,-4)
to [out = down, in =right](0.5,-4.5);
\strand[thick](A.south)
to [out = down, in = left](1.4,-0.4);
\strand[thick](1.6,-0.6)
to [out=down, in=up, looseness =0.7](2,-1)
to [out = down, in = up](1,-2)
to [out = down, in =up](0,-3)
to [out = down, in =up](0,-4)
to [out = down, in =left](0.5,-4.5);
\strand[thick](0.5,-4.5)
to [out = down, in = up](0.5,-5);
\strand[thick](2,0)
to [out= down, in =up](1,-1)
to [out = down, in = up](2,-2)
to [out =down, in = left](2.5,-2.5);
\strand[line width =1.8,red](3,0)
to [out = down, in =up](3,-2)
to [out =down, in =right](2.5,-2.5);
\strand[line width = 1.8, red](2.5,-2.5)
to [out = down, in =up](2.5,-3.1);
\strand[line width = 1.8, red](2.5,-3.3)
to [out =down, in = up](2.5,-5);
\flipcrossings{2,4,5,6,7};
\end{knot}
\end{scope}
\node at (2.9,-0.4) {$= $};
\begin{scope}[shift={(4,0.5)},scale=0.9]
\begin{knot}[clip width=5,clip radius=8pt]
\strand[line width = 1.8, red](1.5,0)
to [out=down, in=right](1,-0.5);
\strand[thick](0.5,0)
to [out=down, in=left](1,-0.5);
\strand[thick](0,0)
to [out = down,in =up](0,-0.5)
to [out=down, in=up](1.5,-1.5)
to [out = down, in =up](0,-2.25)
to [out = down, in=up](0,-2.5);
\strand[line width = 1.8,red](1,-0.5)
to [out=down, in=up](1,-2.5);
\flipcrossings{2}
\end{knot}
\end{scope}
\end{tikzpicture}
\end{figure}

Finally, if  $m_X$ intertwines $*$ and $\circ$, it is easily seen that \eqref{EqTrivmXk} holds.
\end{proof}

\begin{Rem}\label{RemEve}
Note that the braid-commutation of $m_Y^{k}$ and $m_Y^r$ at $A\times A \times e_A \times X$ already implies the $m_X$-braid relation, which then in turn implies the full braid-commutation relation for $m_Y^k$ and $m_Y^r$ under the presence of the other conditions. 
\end{Rem}

\section{Equivalence between actions of skew braces and braided actions of groups with braiding}\label{SecProofEquiv}

Fix a skew brace $(A,\circ,\cdot)$, and let $r$ be the associated braiding operator given by Theorem \ref{TheoEquivBraidGroupBrace}. We write in the following
\[
r(a,b) = (a\rhd b,a\lhd b).
\]
Recall also the notations $\lambda,\rho$ introduced in \eqref{EqDefLR}.

\begin{Lem}
The following identities hold:
\begin{equation}\label{EqIdLR}
a \rhd b = \lambda_a(b) = a^{-1}\cdot (a\circ b),\qquad a\lhd b = \overline{\rho_{\bar{b}}(\bar{a})} = \overline{(\bar{b}\circ \bar{a})\cdot \bar{b}^{-1}} = \overline{\bar{a}\cdot b}\circ b.
\end{equation}
\end{Lem}
\begin{proof}
The identity $a \rhd b =\lambda_a(b)$ is part of the statement of Theorem \ref{TheoEquivBraidGroupBrace}. On the other hand, $a\lhd b$ is determined by \eqref{EqBraidCom},
\begin{equation}\label{EqDefIdrhd}
(a\rhd b)\circ (a\lhd b) = a\circ b.
\end{equation}
Now using 
\[
\lambda_a(b) = a^{-1}\cdot (a\circ b) = a\circ (\bar{a}\cdot b),
\]
we see that \eqref{EqDefIdrhd} becomes
\[
(\bar{a}\cdot b)\circ (a\lhd b) = b,
\]
and so
\[
\overline{a\lhd b} = \bar{b} \circ (\bar{a}\cdot b) = (\bar{b}\circ \bar{a})\cdot \bar{b}^{-1} = \rho_{\bar{b}}(\bar{a}).
\]
The equivalent expression for $a\lhd b$ follows from applying the fundamental skew brace relation. 
\end{proof}

%\qquad a\rhd b = \lambda_a(b) =  a^{-1}\cdot (a\circ b).
%\]

Recall now from \eqref{EqDefTwistProd} the twisted Cartesian product group $A\twistprod A$.  We leave the proof of the following easy lemma to the reader.

\begin{Lem}\label{LemTwistProd}
Let $Y$ be a set. There is a one-to-one correspondence between 
\begin{enumerate}
\item braid-commuting actions $m_Y$ and $m_Y'$ of $(A,\circ)$ on $Y$, and
\item actions $\theta_Y$ of $A\twistprod A$ on $Y$,
\end{enumerate}
the correspondence being given by 
\[
\theta_Y = m_Y'(\id_A\times m_Y).
\]
\end{Lem}

Recall from the discussion at the end of Section \ref{SecBraidGroup} that we have a short exact sequence 
\[
\{e_A\} \rightarrow (A,\cdot) \overset{\iota}{\rightarrow} (A\twistprod A,\circ_{\twist}) \overset{m_A}{\rightarrow} (A,\circ) \rightarrow \{e_A\},
\]
where $\iota(a) = (a,\bar{a})$. In particular, $\iota(A)$ is a normal subgroup of $A\twistprod A$. On the other hand, we also have the group imbedding
\[
j_1: A \rightarrow A\twistprod A,\qquad a \mapsto (a,e_A).
\]
Since 
\[
j_1(a)\circ_{\twist}\iota(b) = (a\circ b,\bar{b}),
\]
the map 
\[
A\times A \rightarrow A\times A,\quad (a,b)\mapsto j_1(a)\circ_{\twist} \iota(b)
\]
is surjective. An easy computation shows that 
\begin{equation}\label{EqInterchange}
j_1(b)\circ_{\twist} \iota(a) = \iota(\rho_b(a))\circ_{\twist} j_1(b),
\end{equation}
%and hence 
%\[
%j_1(b)\circ_{\twist} \iota(a)\circ_{\twist} j_1(\bar{b})= \iota(\gamma_b(a)),
%\]
%showing that $\gamma$ is a group homomorphism
%\[
%\gamma: (A,\circ) \rightarrow \Aut(A,\cdot),\quad a\mapsto \gamma_a.
%\]
and we obtain in this way a group isomorphism
\begin{equation}\label{EqCrossProd}
(A,\circ) \ltimes_{\gamma} (A,\cdot) \cong A\twistprod A,\quad (a,b) \mapsto j_1(a)\circ_{\twist} \iota(b).
\end{equation}

\begin{Lem}\label{PropCCpi}
Let $(X,m_X)$ be an action of $(A,\circ)$. There is a one-to-one correspondence between
\begin{itemize}
\item actions 
\[
m_Y: A\times (A\times X)\rightarrow A\times X,\qquad (a,b,x) \mapsto a*(b,x)
\] 
of $(A,\circ)$ on $Y = A\times X$ which braid-commute with $m_Y^{\triv}$ and for which $m_X: Y \rightarrow X$ is $(A,\circ)$-equivariant, and 
\item maps
\[
\beta: A\times X \rightarrow A,\quad (a,x)\mapsto \beta_x(a)
\]
such that, for each $x\in X$, the map
\begin{equation}\label{EqDefrhox}
\gamma_x:A\times A \rightarrow A,\quad (a,b)\mapsto a\cdot_x b := b\circ \beta_{\bar{b}\circ x} (\rho_{\bar{b}}(a))
\end{equation}
defines an action of $(A,\cdot)$ on $A$, 
\begin{equation}\label{EqActProprhox}
(a\cdot b)\cdot_xc = a\cdot_x(b\cdot_x c),\quad e_A\cdot_x a = a,
\end{equation}
\end{itemize}
The correspondence is completely determined by 
\begin{equation}\label{EqDetermineEquiv}
a*(b,\bar{b}\circ x) = (a\circ (\bar{a}\cdot_x b),-).
\end{equation}
\end{Lem}
\begin{proof}
By Lemma \ref{LemTwistProd}, there is a one-to-one correspondence between, on the one hand, the $m_Y$ braid-commuting with $m_Y^{\triv}$ and making $m_X$ equivariant, and, on the other hand, the $A\twistprod A$-actions
\[
\circledcirc: (A\twistprod A) \times (A\times X) \rightarrow A\times X,\quad ((a,b),(c,x)) \mapsto (a,b)\circledcirc (c,x)
\]
satisfying 
\begin{equation}\label{EqTrivFirstComp}
(a,e_A)\circledcirc (c,x) = (a\circ c,x),
\end{equation}
\begin{equation}\label{EqTrivEquivar}
m_X((a,b)\circledcirc (c,x)) = a\circ b\circ c \circ x,
\end{equation}
the correspondence being determined by 
\begin{equation}\label{EqDetermineEquivBraid}
(a,b)\circledcirc (c,x) = a\circ (b*(c,x)),
\end{equation}
where we write $a\circ (d,y) = (a\circ d,y)$ as a shorthand. 

By \eqref{EqCrossProd} and \eqref{EqTrivFirstComp}, such an action $\circledcirc$ is completely determined by its restriction to $\iota(A)$. Let us write
\[
\gamma: A\times (A \times X) \rightarrow A\times X,\quad (a,b,x) \mapsto \gamma(a)(b,x) = \iota(a)\circledcirc (b,x)
\]
for the corresponding $\cdot$-action. From \eqref{EqTrivEquivar} and the fact that $\iota(A) = \Ker(m_A)$, it follows that the orbits of $\gamma$ are contained in the sets
\[
A_x = \{(a,\bar{a}\circ x)\mid a\in A,x\in X\} \cong A,\quad (a,y) \mapsto a.
\]
Hence $\gamma$ splits into actions
\begin{equation}\label{EqDefRhox}
\gamma_x: A\times A \rightarrow A,\quad (a,c,x)\mapsto \gamma_x(a)(c) = a\cdot_x c
\end{equation}
of $(A,\cdot)$ on $A$, uniquely determined by 
\begin{equation}\label{EqDefPropRho}
\iota(a)\circledcirc (c,\bar{c}\circ x) = (a\cdot_x c, \overline{a\cdot_x c} \circ x).
\end{equation}
Conversely, from \eqref{EqInterchange} and the observation
\[
j_1(b) \circledcirc (c,\bar{c}\circ x) = (b\circ c,\overline{b\circ c}\circ (b\circ x)),
\] 
it is easily seen that a family of actions $\gamma_x$ as in \eqref{EqDefRhox} can be uniquely extended to an action of $A\twistprod A$ satisfying \eqref{EqDefPropRho} and \eqref{EqTrivFirstComp} if and only if for all $a,b,c \in A$ and $x\in X$ 
\begin{equation}\label{EqCompdotxcirc}
b\circ (a\cdot_x c) = \rho_b(a)\cdot_{b\circ x} (b\circ c)
\end{equation}
or, equivalently, 
\begin{equation}\label{EqCompdotxcircVar}
a \cdot_x (b\circ c) = b\circ (\rho_{\bar{b}}(a)\cdot_{\bar{b}\circ x} c),
\end{equation}
condition \eqref{EqTrivEquivar} then being automatic by the form of \eqref{EqDefPropRho}. 

We have thus shown that actions $m_Y$ as in the statement of the lemma are equivalent to families of $(A,\cdot)$-actions 
\[
\gamma_x:A\times A \rightarrow A,\quad (a,b)\mapsto a\cdot_x b
\]
satisfying the compatibility relation \eqref{EqCompdotxcirc}. By \eqref{EqDefPropRho}, the correspondence is uniquely determined by \eqref{EqDetermineEquiv}. 

Define now
\[
\beta: A\times X \rightarrow A,\quad (a,x)\mapsto \beta_x(a):= a\cdot_x e_A.
\]
Taking $c = e_A$ in \eqref{EqCompdotxcircVar}, we see that
\begin{equation}\label{EqDefPropdotx}
a\cdot_x b = b\circ \beta_{\bar{b}\circ x} (\rho_{\bar{b}}(a)).
\end{equation}
Conversely, if 
\[
\beta: A\times X \rightarrow X,\quad (a,x)\mapsto \beta_x(a)
\] 
is any map, then defining $a \cdot_x b$ by means of \eqref{EqDefPropdotx}, we see that \eqref{EqCompdotxcirc} is satisified since
\begin{eqnarray*}
b\circ (a\cdot_x c) &=& b\circ c\circ \beta_{\bar{c}\circ x}(\rho_{\bar{c}}(a))\\
&=& (b\circ c) \circ \beta_{\overline{b\circ c}\circ (b\circ x)}(\rho_{\overline{b\circ c}}(\rho_b(a))\\
&=& \rho_b(a)\cdot_{b\circ x} (b\circ c).
\end{eqnarray*}
Hence the only condition on $\beta$ is that the $\cdot_x$-product defined by \eqref{EqDefPropdotx} defines an action of $A$ on $A$ for each $x$. This finishes the proof. 
\end{proof}

\begin{Lem}\label{LemBraidCommSol}
Let $(X,m_X)$ be an action for $(A,\circ)$, and let $\beta: A\times X \rightarrow A$ be a map such that \eqref{EqActProprhox} holds with respect to \eqref{EqDefrhox}. Let $m_Y$ be the corresponding action of $(A,\circ)$ on $Y = A\times X$ determined by \eqref{EqDetermineEquiv}. Then $m_Y$ braid-commutes with $m_Y^r$ if and only if for all $a,b,c\in A$
\begin{equation}\label{EqCondCommBraid}
a\cdot_x (b\cdot c)  = a\cdot b\cdot a^{-1} \cdot (a\cdot_x c).
\end{equation}
\end{Lem} 
\begin{proof}
It is easy to see that, under the condition \eqref{EqTrivmXk}, the braid-commutation
\begin{equation}\label{EqBraidCommYmYr}
m_Y^r(\id_A \times m_Y) (r\times \id_Y) = m_Y (\id_A \times m_Y^r) 
\end{equation}
will hold if and only if the first components of these maps are the same. Applying then \eqref{EqBraidCommYmYr} to $(a,b,c,\bar{c}\circ x)$ and taking first components, we see by using \eqref{EqDefIdrhd}  that \eqref{EqBraidCommYmYr} is equivalent with the identity 
\begin{equation}\label{EqDiffId}
(a\rhd b)\rhd((a\lhd b) \circ (\overline{a\lhd b}\cdot_x c)) = a\circ (\bar{a}\cdot_{b\circ x}(b\rhd c)).
\end{equation}
Now the left hand side of \eqref{EqDiffId} simplifies as follows, %{EqDefIdrhd}{EqIdLR}
\begin{eqnarray}
\nonumber (a\rhd b)\rhd((a\lhd b) \circ (\overline{a\lhd b}\cdot_x c)) &\underset{\eqref{EqIdLR}}{=}& (a\rhd b)^{-1}\cdot ((a\rhd b)\circ (a\lhd b)\circ (\overline{a \lhd b}\cdot x)) \\ 
&\underset{\eqref{EqDefIdrhd}}{=}& \nonumber (a\rhd b)^{-1}\cdot ((a\circ b) \circ (\overline{a\lhd b} \cdot_x c))\\
& \underset{\eqref{EqIdLR}}{=}& \lambda_a(b)^{-1} \cdot ((a\circ b)\circ (\rho_{\bar{b}}(\bar{a}))\cdot_x c))
\\  \nonumber&\underset{\eqref{EqCompdotxcirc}}{=}& \lambda_a(b)^{-1} \cdot (\rho_{a\circ b}(\rho_{\bar{b}}(\bar{a})) \cdot_{a\circ b \circ x} (a\circ b \circ c)) \\ 
\nonumber &=&   \lambda_a(b)^{-1} \cdot (\rho_{a}(\bar{a}) \cdot_{a\circ b \circ x} (a\circ b \circ c)) \\
\label{EqLHS} &=& \lambda_a(b)^{-1} \cdot (a^{-1} \cdot_{a\circ b\circ x}(a\circ b \circ c)).
\end{eqnarray}
On the other hand, by applying $a\circ$ to $e_A = b^{-1}\cdot b$ we have
\[
 a\cdot (a\circ b)^{-1}\cdot a,
\]
and the right hand side of \eqref{EqDiffId} can then be rewritten as follows, 
\begin{eqnarray} 
\nonumber a\circ (\bar{a} \cdot_{b\circ x} (b\rhd c)) 
&\underset{\eqref{EqCompdotxcirc}}{=}& \rho_a(\bar{a}) \cdot_{a\circ b \circ x}(a\circ \lambda_b(c))\\
\nonumber &=& a^{-1} \cdot_{a\circ b \circ x}(a\circ \lambda_b(c)) \\
\nonumber &=& a^{-1}\cdot_{a\circ b \circ x}(a\circ (b^{-1}\cdot (b\circ c))) \\
\nonumber &=& a^{-1} \cdot_{a\circ b\circ x}((a\circ b^{-1})\cdot a^{-1}\cdot (a\circ b\circ c))\\
\nonumber &=& a^{-1}\cdot_{a\circ b\circ x}(a\cdot (a\circ b)^{-1}\cdot (a\circ b\circ c))\\
\label{EqRHS}&=& a^{-1}\cdot_{a\circ b\circ x}(a\cdot \lambda_a(b)^{-1}\cdot a^{-1} \cdot (a\circ b\circ c)).
\end{eqnarray}
Replacing now subsequently $a\circ b \circ x$ by $x$, $c$ by $\bar{b}\circ \bar{a}\circ c$, $a\cdot \lambda_a(b)^{-1}\cdot a^{-1}$ by $b$ and $a$ by $a^{-1}$, we obtain that the equality of \eqref{EqLHS} and \eqref{EqRHS} becomes \eqref{EqCondCommBraid}.
\end{proof}

\begin{Theorem}
Let $(A,\circ,\cdot)$ be a skew brace with associated braiding operator $r$, and let $(X,\circ)$ be an action of $(A,\circ)$. Then there is a one-to-one correspondence between actions $(X,\circ,\pi)$ of $(A,\circ,\cdot)$ and braided actions $(X,\circ,k)$ of $(A,\circ,r)$, determined by
\[
k(a,x) = (\pi_{a\circ x}(a),\overline{\pi_{a\circ x}(a)}\circ a \circ x).
\]
%Moreover, $k$ is non-degenerate if and only if $\chi$ is non-degenerate.
\end{Theorem}

\begin{proof}
By Lemma \ref{PropCCpi} and Lemma \ref{LemBraidCommSol}, braided actions of $(X,\circ,k)$ are in one-to-one correspondence with a collection of maps 
\[
\beta_x: A \rightarrow A
\]
such that $\eqref{EqDefrhox}$ defines an action satisfying \eqref{EqCondCommBraid}. In particular, applying \eqref{EqCondCommBraid} with $c = e_A$ leads to 
\begin{equation}\label{EqDefxProdDot}
a\cdot_x b = a\cdot b \cdot a^{-1} \cdot \beta_x(a).
\end{equation}
Since $\cdot_x$ is an action, the identity $a\cdot_x(b\cdot_x c) = (a\cdot b)\cdot_xc$ leads to
\begin{equation}\label{EqIdProdBet}
\beta_x(a\cdot b) = a\cdot \beta_x(b)\cdot a^{-1}\cdot \beta_x(a).
\end{equation}
On the other hand, equating \eqref{EqDefxProdDot} with \eqref{EqDefrhox}, we see that $\beta_x$ satisfies the identity
\begin{equation}\label{EqOtherEq}
b\circ \beta_{\bar{b}\circ x}(\rho_{\bar{b}}(a)) = a\cdot b\cdot a^{-1} \cdot \beta_x(a).
\end{equation}

Conversely, assume that $\beta_x: A \rightarrow A$ are a collection of maps satisfying \eqref{EqIdProdBet} and  \eqref{EqOtherEq}. The latter identity tells us that $a\cdot_x b$, as defined by $\eqref{EqDefrhox}$, is given by the formula \eqref{EqDefxProdDot}, and the former identity clearly implies that $\cdot_x$ is then an action. The full identity  \eqref{EqCondCommBraid} then follows immediately from the expression \eqref{EqDefxProdDot} of the $\cdot_x$-action. 

Put now 
\begin{equation}\label{EqDefPiforBet}
\pi_x(a) := \beta_x(a)^{-1}\cdot a,
\end{equation}
so that 
\[
\beta_x(a) = a\cdot \pi_x(a)^{-1}.
\]
It is then sufficient to verify that $\pi_x$ is an endomorphism satisfying \eqref{EqComppicirc} if and only if $\beta_x$ satisfies \eqref{EqIdProdBet} and \eqref{EqOtherEq}. 
%\begin{equation}
%\pi_x(a\circ b) = \lambda_a(\pi_{\bar{a}\circ x}(b))\cdot \pi_x(a),
%\end{equation}
Clearly \eqref{EqIdProdBet} is equivalent to $\pi_x$ being an endomorphism. On the other hand, upon multipliying  \eqref{EqOtherEq} to the left with $b^{-1}\cdot$ and using that $\pi_x$ is an endomorphism, we obtain for $\pi_x$ the further necessary and sufficient condition
\[
\lambda_b(\rho_{\bar{b}}(a) \cdot \pi_{\bar{b}\circ x}(\rho_{\bar{b}}(a^{-1}))) = b^{-1}\cdot a\cdot b \cdot \pi_x(a^{-1}).
\]
Using $\lambda_b(\rho_{\bar{b}}(a)) = b^{-1}\cdot a \cdot b$ and replacing $a$ by $a^{-1}$ and $x$ by $b\circ x$, this simplifies to 
\[
\lambda_b(\pi_{x}(\gamma_{\bar{b}}(a))) = \pi_{b\circ x}(a),
\]
which is the identity \eqref{EqComppicircOther} (with $a$ and $b$ interchanged), and hence equivalent to \eqref{EqComppicirc}.

We have thus shown the one-to-one correpsondence between braided actions $(X,m_X,k)$ and skew brace actions $(X,m_X,\pi)$. The concrete formula for $k$ now follows from observing that
\[
k(a,x) = a*(e_A,x) = (b,\bar{b}\circ a \circ x),
\]
where 
\[
b = a\circ (\bar{a}\cdot_xe_A) \underset{\eqref{EqCompdotxcirc}}{=} \rho_a(\bar{a})\cdot_{a\circ x} a = a^{-1}\cdot_{a\circ x}a \underset{\eqref{EqDefxProdDot}}{=} a \cdot \beta_{a\circ x}(a^{-1}) = \pi_{a\circ x}(a). 
\]
\end{proof}

\begin{Exa}
If $k = r^2$, the associated family of $\pi_b: A \rightarrow A$ is given by the standard skew brace action from Example \ref{ExaStandard},
\[
\pi_b(a) = \alpha_b(a) =  b^{-1}\cdot a \cdot b.
\]
Indeed, we have by construction that $\pi_b(a)$ is the first leg of $k(a,\bar{a}\circ b)$. Now using \eqref{EqIdLR} and \eqref{EqDefIdrhd}, we have
\[
r(a,\bar{a}\circ b) = (\lambda_a(\bar{a}\circ b),\overline{\lambda_a(\bar{a}\circ b)}\circ b) = (a^{-1}\cdot b,\overline{a^{-1}\cdot b} \circ b).
\]
Hence the first leg $\pi_b(a)$ of $k(a,\bar{a}\circ b) = r^2(a,\bar{a}\circ b)$ is given by 
\[
\pi_b(a) = \lambda_{a^{-1}\cdot b}(\overline{a^{-1}\cdot b}\circ b) = (a^{-1}\cdot b)^{-1} \cdot b = b^{-1}\cdot a \cdot b.
\]
\end{Exa}

%We now want to express in terms of $A\twistprod A$ the above actions $m_X$ and $\mu$ of $A$ on respectively $X$ and $A\times X$.

\end{document}